\documentclass[a4paper,12pt]{article}
\usepackage[top=1in,left=1in,right=1in,bottom=1.5in]{geometry}
\usepackage{latexsym}
\usepackage{epsfig, ecltree, epic, eepic}
\usepackage{enumerate,amsmath,amssymb,dsfont,pifont,amsthm}
\usepackage{color}
\usepackage{graphicx}
\usepackage[arrow, matrix, curve]{xy}

\theoremstyle{plain}
\newtheorem{theorem}{Theorem}[section]
\newtheorem{corollary}[theorem]{Corollary}
\newtheorem{lemma}[theorem]{Lemma}
\newtheorem{proposition}[theorem]{Proposition}

\theoremstyle{definition}
\newtheorem{definition}[theorem]{Definition}

\theoremstyle{remark}
\newtheorem{remark}[theorem]{Remark}
\newtheorem{example}[theorem]{Example}

\newcommand{\bbc}{\mathbb{C}}
\newcommand{\bbr}{\mathbb{R}}

\newcommand{\bbp}{\mathbb{P}}
\newcommand{\bbe}{\mathbb{E}}

\newcommand{\bbn}{\mathbb{N}}

\newcommand{\cm}{\mathcal{M}}

\newcommand{\abs}[1]{\left| #1 \right|}
\newcommand{\norm}[1]{\left\| #1 \right\|}

\newcommand{\roi}[2]{\left[ #1 , #2 \right[}

\newcommand{\rosi}[2]{\roi{[#1}{#2[}}

\begin{document}

\allowdisplaybreaks

\title{\bfseries Criteria for the Finiteness of the Strong $p$-Variation for L\'evy-type Processes}

\author{%
    \textsc{Martynas Manstavi\v cius}%
    \thanks{Vilnius University Faculty of Mathematics and Informatics, Naugarduko 24, LT-03225,
              Vilnius, Lithuania and Vilnius University Institute of Mathematics and Informatics, Akademijos 4, LT-08663 Vilnius, Lithuania,
                      \texttt{martynas.manstavicius@mif.vu.lt}. Partially supported by the Research Council of Lithuania, grant No. MIP-53/2012 }
    \textrm{\ \ and\ \ }
    \stepcounter{footnote}\stepcounter{footnote}\stepcounter{footnote}
    \stepcounter{footnote}\stepcounter{footnote}%
    \textsc{Alexander Schnurr}%
    \thanks{Department Mathematik, Universit\"at Siegen, D-57068 Siegen and Fakult\"at f\"ur Mathematik, Technische Universit\"at Dortmund,
              D-44227 Dortmund, Germany,
              \texttt{schnurr@mathematik.uni-siegen.de}. Partially supported by the DFG (German Science Foundation), grant No. SCHN 1231/1-1}
    }
\date{\today}

\maketitle
\begin{abstract}
Using generalized Blumenthal--Getoor indices, we obtain criteria for the finiteness of the $p$-variation of L\'evy-type processes. This class of stochastic processes includes solutions of Skorokhod-type stochastic differential equations (SDEs), certain Feller processes and solutions of L\'evy driven SDEs. The class of processes is wider than in earlier contributions and using fine continuity we are able to handle general measurable subsets of $\bbr^d$ as state spaces.
Furthermore, in contrast to previous contributions on the subject, we introduce a local index in order to complement the upper index. This local index yields a sufficient condition for the infiniteness of the $p$-variation.
We discuss various examples in order to demonstrate the applicability of the method.
\end{abstract}



\section{Historical Introduction}

In 1961, Blumenthal and Getoor \cite{blumenthalget61} defined several indices to control the abundance of small jumps of L\'evy processes and investigate properties of their sample paths. Among the investigated properties were local H\"older conditions, boundedness of $p$-variation and Hausdorff--Besicovitch dimension of interesting random sets. Blumenthal and Getoor indices soon became useful tools for characterizing dichotomies of sample path behavior and motivated the search for even more refined results.
In \cite{pruitt81}, Pruitt used generalizations of the Blumenthal--Getoor indices when analyzing sample path growth properties of a given L\'evy process. R.L. Schilling \cite{schilling98} generalized these results to Feller processes, satisfying the following three conditions \eqref{rich}, \eqref{growth} and \eqref{sector} which we state below since they play a r\^ole in our considerations, too. First of all, the process has to be rich, that is, the test functions are contained in the domain of the generator
\begin{align} \tag{R} \label{rich}
C_c^\infty(\bbr^d)\subseteq D(A).
\end{align}
Using a classical result due to Ph. Courr\`ege \cite{courrege}, N.~Jacob \cite[Chapter 1]{nielsold} showed that the generator $A$ of a process of this kind can be written in the following way:
\[
  Au(x)= -\int_{\bbr^d} e^{ix'\xi} q(x,\xi) \widehat{u}(\xi) \ d\xi   \hspace{10mm} \text{for } u\in C_c^\infty(\bbr^d)
\]
where $\widehat{u}$ is the Fourier transform of $u$ and the symbol $q:\bbr^d \times \bbr^d \to \bbc$ has the following properties:
it is locally bounded in $x,\xi$,
$q(\cdot,\xi)$ is measurable for every $\xi\in\bbr^d$ and
$q(x,\cdot)$ is a continuous negative definite function in the sense of Schoenberg for every $x\in\bbr^d$.
The last point means that the symbol admits a `state-space dependent' L\'evy--Khinchine formula (see formula \eqref{symbol} below).
Let us remark that every L\'evy process is rich Feller and that the symbol coincides with the characteristic exponent in this case.
Therefore, it is not surprising that the r\^ole of the characteristic exponent in the context of L\'evy processes is taken over by the
symbol in the case of rich Feller processes. The two other assumptions in \cite{schilling98} are the growth condition
\begin{align} \tag{G} \label{growth}
\norm{q(\cdot,\xi)}_\infty \leq c(1+\norm{\xi}^2)
\end{align}
and, for some of the results, the sector condition
\begin{align} \tag{S} \label{sector}
\abs{\Im q(x,\xi)}\leq c_0 \Re q(x,\xi);
\end{align}
the constants $c$ and $c_0$ are assumed independent of $x$ and $\xi$.
In \cite{mydiss}, A. Schnurr has shown that every rich Feller process is an \emph{It\^o process} in the sense of E. Cinlar, J. Jacod, Ph. Protter and M.~J. Sharpe
(cf. \cite{vierleute}), that is, a Hunt semimartingale with characteristics of the form
\begin{align} \begin{split} \label{hdwj}
  B_t(\omega) &=\int_0^t  \ell(X_s(\omega)) \ ds,  \\
  C_t(\omega)  &=\int_0^t Q(X_s(\omega)) \ ds,     \\
  \nu(\omega;ds,dy) &=N(X_s(\omega),dy) \ ds
\end{split} \end{align}
where, for every $x\in\bbr^d$, $\ell(x)$ is a vector in $\bbr^d$, $Q(x)$ is a positive semi-definite matrix and $N$ is a Borel transition kernel such that $N(x,\{0\})=0$. Since the term `It\^o process' is,  in fact, used for various different classes, we will use the term \emph{L\'evy-type process} for this kind of stochastic process in the present paper.

The triplet $(\ell(x),Q(x), N(x,dy))$ can be found  in the symbol as well. The symbol -- like the characteristics -- contains a lot of information
about the global and the path properties of the process, like conservativeness \cite{schilling98pos},
the Hausdorff-dimension \cite{schilling98hdd} and the (strong) $p$-variation (cf. \cite[Sect.~5.4]{levymatters3} in the context of rich Feller processes).

It is this latter fine property which we will analyse in the present paper. Let us recall the notion of the $p$-variation which for real-valued functions was given by N.~Wiener back in 1924, developed by L.~C.~Young and E.~R.~Love in the late 1930s and by many other authors since then; for an extensive bibliography on the subject, including research on the boundedness of the $p$-variation of paths of stochastic processes, see \cite[Part IV]{DudNor}. If $p\in]0,\infty[$ and $g$ is an $\bbr^d$-valued function on the interval $[a,b]$ then
\[
  V^p(g; [a,b]) := \sup_{\pi_n} \sum_{j=1}^n \norm{g(t_j)-g(t_{j-1})}^p
\]
where the supremum is taken over all finite partitions $\pi_n =\{t_i\}_{i=1}^n, a=t_0 <t_1 < \cdots < t_n =b$, of $[a,b]$ is called the \emph{(strong) $p$-variation} of $g$ on $[a,b]$.

Our main results can be summarized in short as follows:

For certain generalizations of the Blumenthal--Getoor--Pruitt indices, which we denote by  $\beta_\infty^{unif}$ and $\beta_{loc}^x$, computed for a L\'evy-type process $X$,  the following is true for every $T>0$:
\begin{itemize}
\item if $p > \beta_\infty^{unif}$ then
$
  V^p(X^x;[0,T]) < \infty\quad \bbp^x\text{-almost surely, for all } x\in\bbr^d,
$
\item and if $p< \beta_{loc}^x$ for a fixed $x\in\bbr^d$ and the symbol satisfies some additional regularity conditions then
$  V^p(X^x;[0,T]) = \infty\quad \bbp^x\text{-almost surely}.
$
\end{itemize}
If $X$ is a L\'evy process, then we recover known results of Blumenthal and Getoor \cite[Theorems 4.1 and 4.2]{blumenthalget61} and Monroe \cite[Theorem 2]{Monroe1972}.

One of the main tools we are using to obtain the first part of the result is due to M.~Manstavi\v cius (see \cite{manstavicius}). It relates the $p$-variation of Markov processes with values in a complete separable metric space with the behavior of the transition functions. The details are given in Theorem~\ref{thm:martynas}. This tool is used as well in the recent book by B\"ottcher et al. \cite{levymatters3} in the context of rich Feller processes. Generalizing their result (Proposition~5.21), we do not demand that the process admits the Feller property. Furthermore, the state space under consideration can be more general in our considerations. In the context of this lower bound for $p$, partial results (for Feller processes) are known and we can rely on deep results. This is not the case for the upper bound, that is, our criterion for the infiniteness of $p$-variation. Here, only results for the -- very restrictive class of -- L\'evy processes were known and we had to dev!
 elop new techniques of proof in order to establish our result.

The other main tool we are using is a generalized (probabilistic) symbol: in \cite{nielsursprung}, N.~Jacob showed that under some technical assumptions the following formula can be used to calculate the symbol of a Feller process directly, without writing down the semigroup and calculating the generator:
\begin{align}\label{startingpoint}
   q(x,\xi)=- \lim_{t\downarrow 0}\bbe^x \frac{e^{i(X_t-x)'\xi}-1}{t}.
\end{align}
For conservative processes this can be interpreted as the right-hand side derivative of the state-space-dependent characteristic function at zero. Formula \eqref{startingpoint} was generalized by R.L. Schilling \cite{schilling98pos} to Feller processes satisfying \eqref{rich} and \eqref{growth}. Using a first exit time $\sigma,$ A. Schnurr generalized it further to the class of L\'evy-type processes in the sense of \cite{vierleute} with differential characteristics which are finely continuous (cf. the monograph \cite{fuglede}) and locally bounded. For the details consult Sect. 2.

Various criteria have been set up to check the finiteness of $p$-variation (or more general $\Phi$-variation) of sample paths of stochastic processes; see, e.g. \cite[Ch. 12]{DudleyNorvaisaCFC}. The advantage of our approach is that it is very applicable and yields new results about many interesting processes encountered in applications.

The paper is organized as follows: in the subsequent section, we introduce several indices related to the symbol and discuss their relation to $p$-variation. In Sect.~3, we consider L\'evy driven SDEs, Hunt semimartingales and provide applications of our main results, Theorems \ref{thm:firstcrit}, \ref{thm:lowerbound} and \ref{thm:levysde}, as well as Proposition \ref{prop:unboundedcase}, to many L\'evy-type processes encountered in practice like stable-like, generalized Ornstein--Uhlenbeck, stochastic volatility processes of Barndorff-Nielsen and Shephard as well as COGARCH processes.

Several notation conventions are as follows:
Vectors are column vectors, and $'$ denotes a transposed vector or matrix. In defining new objects, we write `:=' where the object to be defined stands on the left-hand side.
We use the notation $B_r(x):=\{y\in\bbr^d: \norm{y-x}\leq r\}$ for the closed ball of radius $r>0$ centred at $x\in\bbr^d$.
And we write as usual
\[
(X_\cdot-x)_t^*:=\sup_{s\leq t} \norm{X_s-x}
\]
for the maximum process.

\section{Uniform Indices and Their Relation to $p$-Variation}


   We will consider a time homogeneous Markov process $(\Omega, \mathcal{F}, (\mathcal{F}_t)_{t\ge0}, (X_t)_{t\ge0}, \bbp^x)_{x\in\bbr^d}$ with state space $\bbr^d$ which is normal, that is, $\bbp^x(X_0=x)=1$ for all $x\in\bbr^d$, and conservative, i.e., $\bbp^x(X_t\in\bbr^d)=1$ for all $t\ge0$ and $x\in\bbr^d$. The most general assumption  on the differential characteristics of the process, under which we can prove our results, is the so-called fine continuity (cf. \cite[Section II.4]{blumenthalget} and  the monograph \cite{fuglede}). Probably the simplest way to characterize the fine continuity of $f$ is to say that $f\circ X$ is $\bbp^x$-a.s. right continuous for every $x\in\bbr^d$. It is important to note that this property is weaker than ordinary continuity. In order to obtain the upper bound, we combine the result on the strong $p$-variation of M.~Manstavi\v cius \cite{manstavicius} with a maximal inequality of A.~Schnurr \cite{generalsymbindices}.


\begin{definition} \label{def:symbol}
Let $X$ be a Markov semimartingale, which is conservative and normal. Fix a starting point $x$ and define $\sigma=\sigma^x_K$ to be the first exit time from a compact neighbourhood $K:=K_x$ of $x$, namely $\sigma:=\inf\{t\geq 0 : X_t^x \notin K \}$.
For $\xi \in \bbr^d$ we call $q:\bbr^d\times \bbr^d\to \bbc$ given by
\begin{align} \label{stoppedsymbol}
     q(x,\xi):=- \lim_{t\downarrow 0}\bbe^x \frac{e^{i(X^\sigma_t-x)'\xi}-1}{t}
\end{align}
the \emph{symbol} of the process if the limit exists and is independent of the choice of $K$.
\end{definition}

The symbol exists for every L\'evy-type process with locally bounded and finely continuous differential characteristics (see \cite[Theorem 4.4]{mydiss}). A generalization of this result to non-Markov processes can be found in \cite[Theorem 3.6]{generalsymbindices}. The symbol is a state-space dependent continuous negative-definite function in the sense of Schoenberg (cf. \cite{bergforst}) and coincides with the classical symbol in the case of rich Feller processes. For a L\'evy-type process with triplet  $(\ell(x),Q(x), N(x,dy))$ the symbol looks as follows:
\begin{align}\label{symbol}
  q(x,\xi)=-i\ell(x)'\xi + \frac{1}{2} \xi'Q(x) \xi -\int_{y\neq 0} \Big(e^{iy'\xi}-1 -iy'\xi\cdot\chi(y)\Big) \ N(x,dy),
\end{align}
where $\xi=\xi_r:\bbr^d\to\bbr$ is a measurable cut-off function such that $\mathbf{1}_{B_r(0)}\le \xi_r\le \mathbf{1}_{B_{2r}(0)}$ for some $r>0$.

In \cite[Def. 3.8]{generalsymbindices}, A. Schnurr has generalized the eight indices introduced by R.L. Schilling and used these indices in order to derive growth and H\"older conditions of the process. When dealing with the strong $p$-variation, we introduce a uniform index at infinity in order to obtain the global boundedness of the transition probabilities.

\begin{definition}

For $x\in\bbr^d$ 
and
$R>0$, we define the following quantities:
\begin{align}
H(x,R)&:= \sup_{\norm{y-x}\leq 2R} \sup_{\norm{\varepsilon}\leq 1} \abs{q\left(y,\frac{\varepsilon}{R}\right)}, \label{def:H(x,R)}\\
H(R)&:=\sup_{x\in\bbr^d} H(x,R)= \sup_{y\in\bbr^d} \sup_{\norm{\varepsilon}\leq 1} \abs{q\left(y,\frac{\varepsilon}{R}\right)}.
\end{align}
If we want to emphasize the dependence on the symbol $q$, we write $H_q(R)$. There are at least two 
ways to define the \emph{uniform index at infinity} as follows: 
\begin{itemize}
\item
$
\beta_\infty^{unif}:=\inf \left\{\lambda > 0 : \limsup_{R\to 0+} R^\lambda H(R) =0 \right \},
$
\item $
\beta_\infty^{unif, 1}:=\sup_{x\in\bbr^d}\beta_\infty^x:=\sup_{x\in\bbr^d}\inf \left\{\lambda > 0 : \limsup_{R\to 0+} R^\lambda H(x,R) =0 \right \}.
$
\end{itemize}
When needed, we write $\beta_\infty^{unif}(X)$, where $X$ is the stochastic process with this index, and similarly for other indices. Note that $\beta_\infty^{unif,1}\le \beta_\infty^{unif}$. Example \ref{example:1} shows that this inequality can be strict.

In order to obtain the lower bound, we introduce the following {\it local (Blumenthal--Getoor) index}:
\[
\beta_{loc}^x:=\sup_{R>0}\inf_{y\in B_R(x)} \beta^{y},
\]
where
\begin{equation}\label{spot beta}
\beta^{y}:=\inf\left\{\delta>0 \,:\, \lim_{\norm{\xi}\to\infty}\norm{\xi}^{-\delta}\Re q(y,\xi)=0\right\}
\end{equation}
is the ({\it spot}) Blumenthal--Getoor upper index at the point $y$. Clearly, $\beta^x_{loc}\le \beta_\infty^{unif,1}$ for any $x\in\bbr^d$.
\end{definition}
\begin{remark}\label{rem:indices_Levy case} For a L\'evy process $X$, the indices $\beta_\infty^{unif}$, $\beta_\infty^{unif,1}$ and $\beta_{loc}^x$ all reduce to an appropriate form of the the classical Blumenthal--Getoor upper index $\beta$. More precisely, Blumenthal and Getoor \cite{blumenthalget61} define
\[\beta:=\inf\left\{\alpha'>0:\ \int_{||x||<1}||x||^{\alpha'}\nu(dx)<+\infty\right\},
\]
where $\nu(\cdot)$ is the L\'evy measure of the process $X$.
In Theorem 3.2, they show that (i) if $X$ has no Gaussian component and $\beta\ge1$, then $\beta=\beta_1$ where
\[
\beta_1:=\inf\{\alpha'>0: \ \lim_{y\to\infty}\norm{y}^{-\alpha'}|\psi(y)|=0\};
\]
and (ii) if $\beta<1$, $X$ has neither Gaussian  nor linear component, then also $\beta=\beta_2=\beta_1$ where
\begin{equation}\label{beta2}
\beta_2:=\inf\{\alpha'>0: \ \lim_{y\to\infty}\norm{y}^{-\alpha'}\Re\psi(y)=0\}.
\end{equation}
For a general L\'evy process $X$, we have $\beta_\infty^{unif}=\beta_\infty^{unif,1}=\beta_1$, $\beta^x_{loc}=\beta_2$ for all $x\in\bbr^d$; of course, it is possible that $\beta=\beta_2<1= \beta_1$, e.g. for a symmetric pure jump $\alpha$-stable L\'evy process  with $\alpha\in(0,1)$ and an added linear term where  $\beta=\beta_2=\alpha<1=\beta_1$. On the other hand, if a L\'evy process $X$ possesses a Gaussian component, then  $2=\beta_\infty^{unif}=\beta_\infty^{unif,1}=\beta_1=\beta_2=\beta^x_{loc}$ for any $x\in\bbr^d$, while any $\beta\in[0,2]$ is possible. Moreover, for symmetric $\alpha$-stable processes  $\beta_\infty^{unif}=\beta^x_{loc}=\alpha$.

When looking for the $p$-variation index $v(X,[0,T])$ of a L\'evy process $X$ on  $[0,T]$, for any $T>0$ (i.e. the infimum of those $p$ such that the $p$-variation of $X$ on $[0,T]$, $V^p(X,[0,T])$, is finite almost surely), $\beta_1$ is the best form of the upper Blumenthal--Getoor index as $v(X,[0,T])=\beta_1$, i.e. this index captures the influence of the L\'evy measure, as well as Gaussian and linear terms, on the $p$-variation of $X$. By subtracting the Gaussian  and linear terms (if they are present), we reduce the problem to the investigation of the behaviour of the L\'evy measure, and then $\beta$ becomes sufficient.

 On the other hand, for a L\'evy-type process $X$ started at $x$ with triplet $(\ell(x), Q(x), N(x,dy))$ later we find $\beta^x_{loc}$ useful. It does capture the influence of $Q(x)$ and $N(x,dy)$ on the $p$-variation and mimics the definition of $\beta_2$ rather than of $\beta$. Unfortunately, the influence of $\ell(x)$ gets neglected and sometimes leads to imprecise bounds (see Sect. \ref{OU_subsection}). So the situation when $\ell(x)$ is nonzero and nonconstant requires additional research.
\end{remark}

To illustrate possible difference between $\beta_\infty^{unif}$, $\beta_\infty^{unif, 1}$ and $\beta^x_{loc}$, consider the following example.
\begin{example}\label{example:1} Take a standard Brownian motion $W_t, t\ge0$ on $\bbr$ and let $Y_t^x:=x\exp\{W_t-t/2\}, t\ge0, x\in\bbr$, i.e. the solution to
\[
dY_t^x=Y_{t-}^xdW_t,\quad Y_0^x=x\in\bbr.
\]
It is known that the symbol of $Y_t^x$ exists and is given by $q(x,\xi)=(x\xi)^2/2$; see Sect.~\ref{sec:LevySDE} and \ref{sec:stoch-exp}. Then
\[
H(x,R)=\frac1{2R^2}\sup_{y\in[x-2R,x+2R]}y^2=\frac{\max\{(x-2R)^2, (x+2R)^2\}}{2R^2}\to+\infty\quad\text{as}\ x\to\pm\infty
\]
for any $R>0$. Therefore,  $\beta_\infty^{unif}(Y)=+\infty$,  while $\beta_\infty^{unif,1}(Y)=2$.

On the other hand, $\beta^x(Y)=2$ for all $x\ne0$ and $\beta^0(Y)=0$, hence
\[
\beta_{loc}^x(Y)=\left\{
                \begin{array}{ll}
                  2, & \hbox{if $x\ne 0$;} \\
                  0, & \hbox{if $x=0$.}
                \end{array}
              \right.
\]
 \end{example}

In Sect. \ref{sec:LevySDE}, we show that, in fact, $\beta_\infty^{unif,1}(X)=\beta_\infty^{unif}(X)$ when the latter index is finite and $X$ is
a solution of an SDE driven by a L\'evy process (see (\ref{levysde})), provided the coefficient matrix satisfies a certain nondegeneracy condition.


\begin{remark}
The index $\beta_\infty^{unif}(X)$ can be infinite even for a solution of the SDE \eqref{levysde}. But it is always finite if the growth condition (G) is fulfilled, that is, if the differential characteristics are bounded.
\end{remark}

Next we adapt Proposition 3.10 of \cite{generalsymbindices} to our situation; cf. in this context \cite[Lemma 4.1]{schilling98} and \cite[Corollary 5.2]{levymatters3}.

\begin{proposition} \label{prop:boundinprob}
Let $X$ be a L\'evy-type process such that the differential characteristics $\ell$, $Q$ and $n:=\int_{y\neq 0} (1\wedge \norm{y}^2) \ N(\cdot,dy)$ of $X$ are locally bounded and finely continuous. In this case, we have for every $x\in\bbr^d$
\begin{align} \label{firstestimate}
\bbp^x\Big((X_\cdot - x)_t^* \geq R \Big) \leq c_d \cdot t \cdot H(x,R)
\end{align}
for $t\geq 0$, $R>0$ and a constant $c_d>0$ which can be written down explicitly and only depends on the dimension $d$.
\end{proposition}

\begin{remark}
If the state space is a Borel measurable subset $B\subseteq \bbr^d$, we can still use the proposition above and prove our main results below. This is owed to the fact that fine continuity (and local boundedness) are sufficient for establishing \eqref{firstestimate}. Therefore, it is easy to prolong the state space to $\bbr^d$ by setting $X_t(\omega)=x$ under $\bbp^x$ for each $\omega\in\Omega$, $t\geq 0$ and $x\in B^c$. The differential characteristics of the extended process are zero on $B^c$. Hence, if the requirements for the above proposition are true on $B$, they remain true on $\bbr^d$ and, therefore, Proposition \ref{prop:boundinprob} can be applied. Using Theorem \ref{thm:martynas} on the extended process (which is trivial on $B^c$) yields the result for the process with state space $B$ as well. Summing up, we can prove analogous results in case of the upper bound for general measurable state space $B\subseteq \bbr^d$ which allows us to include, e.g. affine processes (cf. \cite{Duffie_etal2003}) in our considerations.
\end{remark}

\begin{definition}
Let $\nu\ge1$ and $\gamma>0$. We say that a Markov process $X$ on $[0,T]$ belongs to the class $\cm(\nu,\gamma)$ if there exist constants $a_0>0$ and $K>0$ such that, for all $h\in[0,T]$ and $a\in ]0,a_0]$,
\[
\alpha(h,a)\leq K \frac{h^\nu}{a^\gamma}
\]
where
$
\alpha(h,a)=\sup\{\bbp^x(\norm{X_t-x}\geq a) : x\in\bbr^d, 0\leq t \leq h \}.
$
\end{definition}

Let us recall Theorem 1.3 of \cite{manstavicius} which reads in our situation as follows:

\begin{theorem} \label{thm:martynas}
Let $X$ be a strong Markov process on $[0,T]$ with values in $\bbr^d$. Suppose $X\in\cm(\nu,\gamma)$. Then for any $p>\gamma/\nu$ the $p$-variation of $X^x$ is finite on $[0,T]$
$\bbp^x$-almost surely for every $x\in\bbr^d$.
\end{theorem}

\begin{theorem} \label{thm:firstcrit}
Let $X$ be a L\'evy-type process with locally bounded and finely continuous differential characteristics which is conservative and normal.
Let $\beta_\infty^{unif}(X)$ be the generalized index of $X$ and assume this index is finite. For every $T>0$ and $p > \beta_\infty^{unif}(X)$ we obtain
\[
  V^p(X^x;[0,T]) < \infty\quad \bbp^x\text{-almost surely}.
\]

\end{theorem}

\begin{proof}
Fix $T>0$. Let $h\in [0,T]$, $r>0$ and $\lambda>p$. Consider
\begin{align*}
\alpha(h,r)&=\sup\{\bbp^x(\norm{X_t-x}\geq r) : x\in\bbr^d, 0\leq t \leq h \} \\
           &\leq \sup\{\bbp^x((X_t-x)^*\geq r) : x\in\bbr^d, 0\leq t \leq h \} \\
           &\leq \sup\{c_d \cdot t \cdot H(x,r): x\in\bbr^d, 0\leq t \leq h \} \\
           &=c_d \cdot h \cdot H(r) \\
           &\leq c_d \cdot c \cdot (h^1/r^\lambda)
\end{align*}
on $r\in]0,a_0]$ where $a_0$ is chosen appropriately. For the last inequality we have used the definition of the index and $\lambda>\beta_\infty^{unif}(X)$.
 By Theorem \ref{thm:martynas} above, we obtain the result.
\end{proof}

Our result is in the following sense best possible: in the well-known case of symmetric $\alpha$-stable processes, it is known that the $p$-variation is finite for $p>\alpha=\beta_\infty^{unif}$ and that it is infinite for $p=\alpha=\beta_\infty^{unif}$ (see, e.g. \cite{blumenthalget61} and \cite{Monroe1972}).
\begin{remark}
Due to Proposition \ref{prop:unboundedcase} below, Theorem \ref{thm:firstcrit} generalizes Corollary 5.10 of \cite{sdesymbol} which, unfortunately, has a missing condition needed
in its proof, namely, that $\beta_\infty^{unif}(X)<+\infty$, otherwise the bound of $\alpha(h,r)$ might be infinite due to unboundedness of $H(r)$ for at least
a sequence of $r$'s converging to $0$.
\end{remark}

Now we consider the lower bound and
start by presenting the following lemmas:
\begin{lemma}\label{lem:infiniteintegral} Let $X$ be a L\'evy-type process with symbol $q$. Consider a fixed $x\in\bbr^d$ such that $\beta^{x}>0$
then
\begin{equation}\label{eq:D(x) defn}
D(x):=\int_{\bbr^d}f_\lambda(\xi)\Re q(x,\xi)d\xi\left\{\begin{array}{ll}
                                                             =+\infty, & \hbox{if $0<\lambda<\beta^{x}$;} \\
                                                             <+\infty, & \hbox{if $2\ge\lambda>\beta^x$ and $\beta^x<2$}\\
                                                             & \hbox{or if $\lambda=2$ and $\beta^x=2$,}
                                                  \end{array}
                                                  \right.
\end{equation}
where 
 $f_\lambda$ is the density of a symmetric $\lambda$-stable random vector $Y$ such that
\begin{equation}\label{flambda}
e^{-\norm{v}^\lambda}=\mathbb{E} e^{iv'Y}=\int_{\bbr^d}e^{iv'\xi}f_\lambda(\xi)d\xi, \quad v\in\bbr^d, \lambda\in(0,2].
\end{equation}
\end{lemma}

\begin{proof} The claim is essentially that of \cite[Lemma 3.3]{blumenthalget61}; however, a different starting point in defining the index $\beta^x$ requires slight modifications in the proof.

Indeed, if $\beta^x=2$, then for $\lambda=2$ we have $f_2(\xi)=Ae^{-B\norm{\xi}^2}$ for some known (but unimportant for our considerations) positive constants $A$ and $B$ depending on the dimension $d$ and, since
$\Re q(x,\xi)=O(\norm{\xi}^2)$ as $\norm{\xi}\to\infty$ for any fixed $x$, we clearly have convergence of $D(x)$ in this case.

So assume $\beta^x<2$. This in particular implies that $X$ has no Gaussian component and
\[
\Re q(x,\xi)=\int_{\bbr^d} (1-\cos(y'\xi))N(x,dy)
\]
so that, using Fubini theorem and an elementary inequality $1-e^{-z}\ge z/e$ valid for all $z\in[0,1]$,
\[
\begin{split}
D(x)&=\int_{\bbr^d} f_\lambda(\xi)\int_{\bbr^d}(1-\cos(y'\xi))N(x,dy)d\xi\\
&=\int_{\bbr^d}\left(1-e^{-\norm{y}^\lambda}\right)N(x,dy)\\
&\ge\int_{\norm{y}<1}\left(1-e^{-\norm{y}^\lambda}\right)N(x,dy)\\
&\ge \frac1e\int_{\norm{y}<1}\norm{y}^\lambda N(x,dy)
\end{split}
\]
where the right-hand side is infinite if $\lambda<\beta^x$ since, by \cite[Theorem 3.2]{blumenthalget61}, for any fixed $x\in\bbr^d$,
\[
\beta^x=\inf\left\{\alpha>0: \int_{\norm{y}<1}\norm{y}^\alpha N(x,dy)<+\infty\right\}.
\]

On the other hand, if $\lambda\in(\beta^x,2]$, then similarly (since $1-e^{-z}\le 1\wedge z$ for all $z\in[0,\infty)$)
\[
D(x)\le \int_{\bbr^d}1\wedge\norm{y}^\lambda N(x,dy)<+\infty.
\]
\end{proof}

For fixed $\lambda\in(0,2]$ and $R>0$, define the following function for notational convenience:
\begin{equation}\label{h_function}
h(t,y)=h(t,y;\lambda,R):=\int_{\bbr^d}f_\lambda(\xi)\frac{1-\Re\mathbb{E}^ye^{i(X_t^\sigma-y)'\xi}}{t}d\xi\in[0,1/t]
\end{equation}
where $t>0$, $y\in\bbr^d$, and $\sigma=\sigma_{B_R(y)}$ is the first exit time from $B_R(y)$ for a L\'evy-type process $X$ with symbol $q$ which starts at $y$ $\mathbb{P}^y$-a.s..
The function $f_\lambda(\xi)$ is as in Lemma~\ref{lem:infiniteintegral}.
For any $\kappa>0$ and $x\in\bbr^d$, let also
\begin{equation}\label{underline_h_defn}
\underline{h}(t,x)=\underline{h}(t,x;\lambda,\kappa, R):=\inf_{y\in B_\kappa(x)}h(t,y).
\end{equation}
\begin{lemma} Let $h(t,y)$ be defined by \eqref{h_function}. Then for all fixed $\lambda\in(0,2]$ and $R>0$, and any  $y\in\bbr^d$,
\begin{equation}\label{h(t,y)_limit}
\lim_{t\downarrow0} h(t,y)=D(y),
\end{equation}
where the sides are either both finite or both $+\infty$, and $D$ is defined in \eqref{eq:D(x) defn}.
\end{lemma}
\begin{proof} Since the integrand in the definition of $h(t,y)$ is non-negative for all $t>0$ and $y\in\bbr^d$, by using Fatou's lemma and the definition of the symbol, we get
\[
\liminf_{t\downarrow0} h(t,y)\ge\int_{\bbr^d}f_\lambda(\xi)\Re q(y,\xi)d\xi=D(y).
\]
This implies, in particular, that both sides of \eqref{h(t,y)_limit} are $+\infty$ if $0<\lambda<\beta^y\in(0,2]$, by Lemma \ref{lem:infiniteintegral}.

So in the rest of the proof assume that for a fixed $y\in\bbr^d$, $\lambda\in(\beta^y,2]$ (hence $D(y)<+\infty$) and
 consider any sequence $t_n=t_n(y)\downarrow 0$ as $n\to 0$ such that
\[
\limsup_{t\downarrow0}h(t,y)=\lim_{n\to\infty}h(t_n,y).
\]
Observe that the sequence of nonnegative functions
\[
g_n(y,\xi):=g_n(y,\xi;\lambda, R):=f_\lambda(\xi)\frac{1-\Re\mathbb{E}^ye^{i(X_{t_n}^\sigma-y)'\xi}}{t_n},\quad n\in\mathbb{N},
\]
is bounded above by an integrable function $g(y,\xi):=\sum_{i=1}^{n_0}g_i(y,\xi)+f_\lambda(\xi)(\Re q(y,\xi)+1)$, with $||g(y,\cdot)||_{L^1}\le n_0t_{n_0}^{-1}+D(y)+1$, where
$n_0\in\mathbb{N}$ is large enough so that
\[
g_n(y,\xi)\le \sup_{m\ge n} g_m(y,\xi)\le f_\lambda(\xi)(\Re q(y,\xi) +1),\quad \forall n>n_0;
\]
this is possible since, using the definition of the symbol, $\limsup_{n\to\infty}g_n(y,\xi)=\lim_{n\to\infty}g_n(y,\xi)=f_\lambda(\xi)\Re q(y,\xi)$.
Now the choice of $t_n$ and the reverse Fatou's inequality yield
\[
D(y)=\int_{\bbr^d}\lim_{n\to\infty}g_n(y,\xi)d\xi
\ge \limsup_{n\to\infty}h(t_n,y)=\limsup_{t\downarrow0}h(t,y).
\]
Hence the claim of the lemma.
\end{proof}

To control the dependence on the starting position of the process, next we need the following condition:

\noindent{\bf Condition (A$_x$).} L\'evy-type process $X$ with symbol $q$ is such that for a given $x\in\bbr^d$ there are $\delta'=\delta'(x)>0$ and $R'>0$, possibly depending on $\delta'$, such that for all $\Delta>0$
the function $(y,\xi)\mapsto \Re q(y,\xi)$ is jointly continuous on $B_{\delta'}(x)\times B_\Delta(0)$ and
\begin{equation}\label{eq:conditionAx}
\lim_{t\downarrow0}\sup_{(y,\xi)\in B_{\delta'}(x)\times B_\Delta(0)}\left|\frac{1-\Re\mathbb{E}^ye^{i(X_t^\sigma-y)'\xi}}{t}-\Re q(y,\xi)\right|=0,
\end{equation}
where $\sigma=\sigma^y_{B_R(y)}$ for any $R\in (0, R')$,
that is, convergence of real parts (to a continuous limit) in \eqref{stoppedsymbol} is uniform in $y$ in a closed ball around $x$ and on compacts with respect to $\xi$ (at least for small values of $R$).

In various examples, we typically know the expression for the symbol, so checking its continuity (or the continuity of the differential characteristics of $X$) is relatively easy, but checking uniform convergence in \eqref{eq:conditionAx} requires more effort. An easy to check sufficient condition for Condition (A$_x$) to hold is given by Corollary \ref{cor:sufficient_Ax} below, but before we need several results.

\begin{lemma}\label{unif_convergence_means} Let $X$ be a L\'evy-type process as in Proposition \ref{prop:boundinprob}. Fix any $x\in\bbr^d$ and take any $\kappa>0$.  Then for any $\nu>0$,
\[
\lim_{t\downarrow0}\sup_{y\in B_\kappa(x)}\mathbb{E}^y \sup_{0\le s\le t}\Big(\norm{X_s-y}^\nu 1_{[[0,\sigma[[}(s)\Big)=0,
\]
where $\sigma=\sigma^y_{B_R(y)}$ for any fixed $R>0$.
\end{lemma}
\begin{proof} Using the definition of $\sigma$ and Proposition \ref{prop:boundinprob}, for any $t\le R^2$,
\[
\begin{split}
J(t;x,\kappa, R)&:=\sup_{y\in B_\kappa(x)}\mathbb{E}^y \sup_{0\le s\le t}\norm{X_s-y}^\nu 1_{[[0,\sigma[[}(s)\\
&\le\sup_{y\in B_\kappa(x)}\int_0^{R^\nu}\mathbb{P}^y((X_\cdot-y)^*_t>u^{1/\nu})du \\
&\le t^{\nu/2}+c_dt\int_{t^{\nu/2}}^{R^\nu}\sup_{y\in B_\kappa(x)}H(y,u^{1/\nu})du,
\end{split}
\]
where  the constant $c_d>0$ is from \eqref{firstestimate}. Then for any $u\le R^\nu$, using \eqref{def:H(x,R)} and the well-known inequality $\norm{q(\cdot,v)}_K\le c_{q,K}(1+\norm{v}^2)$ with $K=B_{\kappa+2R}(x)$ (see, e.g. \cite[Lemma 6.2]{sdesymbol}),
\begin{equation}\label{bound for H(y,u)}
\sup_{y\in B_\kappa(x)}H(y,u^{1/\nu})\le \sup_{y\in B_{\kappa+2R}(x)}\sup_{\norm{v}\le u^{-1/\nu}}|q(y,v)|\le c_{q,B_{\kappa+2R}(x)}(1+u^{-2/\nu}).
\end{equation}
Hence the result follows from the following estimates:
\[
J(t;x,\kappa,R)\le t^{\nu/2} + c_2\left(R^\nu t+\frac{R^{\nu-2}t-t^{\nu/2}}{1-2/\nu}\mathbf{1}_{\{\nu\ne2\}}+t\ln\left(\frac{R^2}{t}\right)\mathbf{1}_{\{\nu=2\}}\right)\to 0\quad\text{as $t\downarrow0$},
\]
if $\nu>0$, where $c_2:=c_d\cdot c_{q,B_{\kappa+2R}(x)}$.
\end{proof}

\begin{proposition}\label{prop:integral_V_convergence} Let $X$ be a L\'evy-type process as in Proposition \ref{prop:boundinprob}.
Consider a function $V:\bbr^d\times\bbr^d\to \mathbb{C}$ such that for some $x\in\bbr^d$ there exists a $\delta>0$ such that for all $\Delta>0$
the function $V$ restricted to $B_{\delta}(x)\times B_\Delta(0)$ is continuous.
Then for any  $\kappa\in(0, \delta)$, $R\in(0,\delta-\kappa)$, and with $\sigma=\sigma_{B_R(y)}^y$ being the first exit time of $X^y$ from the ball $B_R(y)$,
\[
\lim_{t\downarrow0}\sup_{(y,\xi)\in B_\kappa(x)\times B_\Delta(0)}\mathbb{E}^y\left|\frac1t\int_0^te^{i(X_s^\sigma-y)'\xi} V(X_s,\xi) 1_{[[0,\sigma[[}(s)ds-V(y,\xi)\right|=0.
\]
\end{proposition}
\begin{proof} Elementary estimates allow us to write (for any $y\in B_\kappa(x)$)
\[
\begin{split}
&\mathbb{E}^y\left|\frac1t\int_0^te^{i(X_s^\sigma-y)'\xi} V(X_s,\xi) 1_{[[0,\sigma[[}(s)ds-V(y,\xi)\right|\\
&\le \mathbb{E}^y\sup_{0\le s\le t}\left|e^{i(X_s^\sigma-y)'\xi} V(X_s,\xi) 1_{[[0,\sigma[[}(s)-V(y,\xi)\right|\\
&\le \mathbb{E}^y\sup_{0\le s\le t}\left|V(X_s,\xi)-V(y,\xi)\right| 1_{[[0,\sigma[[}(s)\\
&\ \ +|V(y,\xi)|\left\{\mathbb{E}^y\sup_{0\le s\le t}\left|e^{i(X_s^\sigma-y)'\xi}-1\right| 1_{[[0,\sigma[[}(s)+\mathbb{E}^y\sup_{0\le s\le t}1_{]]\sigma,\infty[[}(s)\right\}\\
&=:I_1(t;y,\xi)+I_2(t;y,\xi).
\end{split}
\]
To show that $I_1(t;y,\xi)\to0$ as $t\downarrow0$ uniformly for $(y,\xi)\in B_\kappa(x)\times B_\Delta(0)$, consider any $\varepsilon>0$. Since, by assumption, $V$ is jointly continuous on
$B_{\delta}(x)\times B_{2\Delta}(0)$, we can consider a sequence $\tilde{V}_{\tilde{\varepsilon}}$ of smooth functions which converge to $V$ as $\tilde{\varepsilon}\downarrow0$ uniformly on $B_\kappa(x)\times B_\Delta(0)$ which is compact in $B_{\delta}(x)\times B_{2\Delta}(0)$. As is standard in the analysis of differential equations, such a sequence is obtained by taking convolutions of $V$ with a smooth mollifier. Hence, we can pick an $\tilde{\varepsilon}_0>0$ small enough that
\[
\sup_{(y,\xi)\in B_\kappa(x)\times B_\Delta(0)}|\tilde{V}_{\tilde{\varepsilon}}(y,\xi)-V(y,\xi)|\le \varepsilon/3,\quad \forall \tilde{\varepsilon}\in(0,\tilde{\varepsilon}_0).
\]
Then by taking any  $\tilde{\varepsilon}\in(0,\tilde{\varepsilon}_0)$, since $R\in(0,\delta-\kappa)$ and $X_s^y\in B_R(y)\subset B_{\delta}(x)$ on the set $\{\sigma^y_{B_R(y)}>s\}$,
\[
\begin{split}
I_1(t;y,\xi)&=\mathbb{E}^y\sup_{0\le s\le t}\left|V(X_s,\xi)-V(y,\xi)\right| 1_{[[0,\sigma[[}(s)\\
&\le \frac{2\varepsilon}3+\mathbb{E}^y\sup_{0\le s\le t}\left|\tilde{V}_{\tilde\varepsilon}(X_s,\xi)-\tilde{V}_{\tilde\varepsilon}(y,\xi)\right| 1_{[[0,\sigma[[}(s)\\
&\le \frac{2\varepsilon}3+C_1\mathbb{E}^y\sup_{0\le s\le t}||X_s-y|| 1_{[[0,\sigma[[}(s),
\end{split}
\]
where $C_1:=\sup_{(z,\xi)\in B_\kappa(x)\times B_\Delta(0)}\big|\frac{\partial}{\partial z}\tilde{V}_{\tilde{\varepsilon}}(z,\xi)\big|<+\infty$. Now applying Lemma \ref{unif_convergence_means},
\[
\sup_{(y,\xi)\in B_\kappa(x)\times B_\Delta(0)}I_1(t;y,\xi)\le \varepsilon,
\]
as soon as $t\le t_0$ for some $t_0>0$.

As for $I_2(t;y,\xi)$, by letting $C_2:=\sup_{(y,\xi)\in B_\kappa(x)\times B_\Delta(0)}|V(y,\xi)|<+\infty$,
\[\begin{split}
\sup_{(y,\xi)\in B_\kappa(x)\times B_\Delta(0)}I_2(y,\xi)&\le C_2\left\{\Delta \sup_{y\in B_\kappa(x)}\mathbb{E}^y\sup_{0\le s\le t}\norm{X_s-y} 1_{[[0,\sigma[[}(s)\right.\\
&\left.\qquad+\sup_{y\in B_\kappa(x)}\mathbb{P}^y\left(\sup_{0\le s\le t}\norm{X_s-y}\ge R\right)\right\}\to0
\end{split}
\]
as $t\downarrow 0$, by Lemma \ref{unif_convergence_means} and Proposition \ref{prop:boundinprob}, since $|e^{iv}-1|\le|v|$ for any $v\in\bbr$ and (see \eqref{bound for H(y,u)})
\[
\sup_{y\in B_\kappa(x)}\mathbb{P}^y\left(\sup_{0\le s\le t}\norm{X_s-y}\ge R\right)\le c_dt\sup_{y\in B_\kappa(x)}H(y,R)\le c_dc_{q, B_{\kappa+2R}(x)}t(1+R^{-2})\to0
\]
as $t\downarrow0$, completing the proof.
\end{proof}

Following the proof of \cite[Lemma 3.4]{Behme_Schnurr2015} or \cite[Theorem 4.2]{mydiss}, we get
\begin{equation}\label{3converging_parts}
\frac{1-\mathbb{E}^ye^{i(X_t^\sigma-y)'\xi}}{t}-q(y,\xi)=\sum_{j=1}^3 \mathbb{E}^y\left(\frac1t\int_0^te^{i(X_s^\sigma-y)'\xi} V_j(X_s,\xi) 1_{[[0,\sigma[[}(s)ds-V_j(y,\xi)\right)
\end{equation}
where
\begin{equation}\label{eq:defVj}
V_j(y,\xi)=\left\{
             \begin{array}{ll}
               i \ell(y)'\xi, & \hbox{if $j=1$;} \\
               -\frac12\xi' Q(y)\xi, & \hbox{if $j=2$;} \\
               \int_{z\ne0}(e^{i\xi'z}-1-i\xi'z\chi(z))N(y,dz), & \hbox{if $j=3$.}
             \end{array}
           \right.
\end{equation}
Thus we have the following corollary of Proposition \ref{prop:integral_V_convergence}
\begin{corollary}\label{cor:sufficient_Ax} Let $X$ be a L\'evy-type process with finely continuous and bounded differential characteristics $\ell$, $Q$ and $\int_{y\ne0}1\wedge\norm{y}^2N(\cdot,dy)$.
If there exist a point $x\in\bbr^d$ and a $\delta>0$ such that $\ell$ and $Q$ are continuous on $B_\delta(x)$ and $y\mapsto V_3(y,\xi)$ (see Eq. \eqref{eq:defVj}) is continuous on $B_\delta(x)$ for any $\xi\in\bbr^d$ then   Condition~(A$_{x}$) holds for any $\delta'\in(0,\delta)$ and $R'\in(0,\delta-\delta')$.
\end{corollary}
\begin{proof} By assumption on $\ell$ and $Q$, $V_1(y,\xi)$ and $V_2(y,\xi)$ are clearly jointly continuous on $B_\delta(x)\times B_\Delta(0)$ for any $\Delta>0$. Joint continuity of $V_3(y,\xi)$ on the same set follows from the continuity of $y\mapsto V_3(y,\xi)$ on $B_\delta(x)$ and continuity of $\xi\mapsto V_3(y,\xi)$ on $B_\Delta(0)$ for any $\Delta>0$ and any $y\in \bbr^d$; the latter comes from dominated convergence.
\end{proof}

The next lemma justifies the interchange of $\liminf$ and $\inf$ needed later on.
\begin{lemma} \label{lem:infexchange} Consider a L\'evy-type process $X$ with symbol $q$ that satisfies Condition~(A$_x$) for some $x\in\bbr^d$, $\delta'>0$ and $R'>0$. Let the functions $\underline{h}(t,x)$  and $D(x)$ be defined by \eqref{underline_h_defn} and \eqref{eq:D(x) defn}, respectively.  Then for any $\lambda\in(0,2]$, $\kappa\in(0, \delta')$, and $R\in(0,R')$
\begin{equation}\label{eq:infexchange}
\liminf_{t\downarrow0}\underline{h}(t,x)=\inf_{y\in B_\kappa(x)}D(y)=:\underline{D}(x),
\end{equation}
where the sides are both finite or both infinite.
\end{lemma}
\begin{proof} First observe that for any $\kappa\in(0,\delta')$ and any $y\in B_\kappa(x)$, we have
\[
\liminf_{t\downarrow0}\underline{h}(t,x)\le \liminf_{t\downarrow0} h(t,y)=D(y)
\]
(see Eqs.  \eqref{underline_h_defn} and \eqref{h(t,y)_limit}), and hence
\[
\liminf_{t\downarrow0}\underline{h}(t,x)\le \underline{D}(x).
\]
So if the left-hand side of \eqref{eq:infexchange} is infinite, we are done.

The essence of the lemma is the opposite inequality when the left-hand side of \eqref{eq:infexchange} is finite.
For the rest of the proof, assume that this is the case. Then start by picking any sequence $t_n\downarrow0$ as $n\to\infty$ such that
\[
\liminf_{t\downarrow0}\underline{h}(t,x)=\lim_{n\to\infty}\underline{h}(t_n,x).
\]
Consider any $\varepsilon>0$. Then for each $n\in\mathbb{N}$ we can find a $y_n=y(\varepsilon, t_n, \kappa,x)\in B_\kappa(x)$ such that
\begin{equation}\label{h(t,y) bound above}
h(t_n,y_n)\le \underline{h}(t_n,x)+\varepsilon.
\end{equation}
Since $B_\kappa(x)$ is compact, by passing to a subsequence of $\{y_n\}$ (and of $\{t_n\}$, respectively) if necessary, we may assume that $y_n\to y_0\in B_\kappa(x)$ as $n\to\infty$.
Consider two cases:
\begin{enumerate}
\item[Case 1.] `$D(y_0)=+\infty$'. Choose any $M>0$ and find a $\Delta>0$ such that
\[
\int_{B_\Delta(0)}f_\lambda(\xi)\Re q(y_0,\xi)d\xi\ge M+3\varepsilon.
\]
By Condition (A$_x$), there is a $t_0=t(x,\kappa,\Delta, R)>0$ such that
\[\sup_{(y,\xi)\in B_{\kappa}(x)\times B_\Delta(0)}\left|\frac{1-\Re \mathbb{E}^y e^{i(X_t^\sigma-y)'\xi}}{t}-\Re q(y,\xi)\right|<\varepsilon
\]
for all $t\in(0,t_0)$ and $(y,\xi)\mapsto \Re q(y,\xi)$ is continuous on $B_{\kappa}(x)\times B_\Delta(0)$. Let $N_0:=\min\{n\in\mathbb{N}: t_m\le t_0,\, \forall m\ge n\}$. Now, $\Re q(y,\xi)$ being continuous is uniformly continuous on $B_{\kappa}(x)\times B_\Delta(0)$ and so there exists an integer $N_1$ such that
\[
|\Re q(y_n,\xi)-\Re q(y_0,\xi)|\le \varepsilon\quad\text{for all}\ n\ge N_1.
\]
So for all $n\ge N_2:=\max\{N_0,N_1\}$ we have
\begin{equation}\label{h(t,y)_bound below}
\begin{split}
h(t_n,y_n)&\ge\int_{B_\Delta(0)}f_\lambda(\xi)\frac{1-\Re \mathbb{E}^{y_n} e^{i(X_{t_n}^\sigma-y_n)'\xi}}{t_n}d\xi\\
&\ge\int_{B_\Delta(0)}f_\lambda(\xi)\Re q(y_n,\xi)d\xi-\varepsilon\int_{B_\Delta(0)}f_\lambda(\xi)d\xi\\
&\ge \int_{B_\Delta(0)}f_\lambda(\xi)\Re q(y_0,\xi)d\xi-2\varepsilon\\
&\ge M+\varepsilon.
\end{split}
\end{equation}
Combining this with \eqref{h(t,y) bound above}, we get
\[
\underline{h}(t_n,x)\ge M\quad\text{for all}\ n\ge N_2,
\]
which, due to arbitrariness of $M$, yields $\lim_{n\to\infty}\underline{h}(t_n,x)=+\infty$ and a contradiction.

\item[Case 2.] `$D(y_0)<+\infty$'. In this case, clearly, $\underline{D}(x)<+\infty$. Now choose a $\Delta>0$ large enough so that
\[
D(y_0)\le \int_{B_\Delta(0)} f_\lambda(\xi)\Re q(y_0,\xi)d\xi+\varepsilon.
\]
As in Case 1, by Condition (A$_x$), we get
\[
h(t_n,y_n)\ge \int_{B_\Delta(0)} f_\lambda(\xi)\Re q(y_0,\xi)d\xi - 2\varepsilon\ge D(y_0)-3\varepsilon,
\]
for all $n\ge N_2$, where $N_2$ is the same as in Case 1. The only difference is how $\Delta$ is chosen.
Now recalling \eqref{h(t,y) bound above},
\[
\underline{D}(x)\le D(y_0)\le \underline{h}(t_n,x)+4\varepsilon,\quad\text{for all } n\ge N_2,
\]
which yields
\[
\underline{D}(x)\le \liminf_{t\downarrow0}\underline{h}(t,x)+4\varepsilon.
\]
Since $\varepsilon$ was arbitrary, letting it decrease to zero, we obtain the sought inequality.
\end{enumerate}
\end{proof}

Now we're ready to state and prove the following theorem:
\begin{theorem}\label{thm:lowerbound} Let $X$ be a L\'evy-type process with symbol $q$. Assume that Condition~(A$_x$) holds for some $x\in\bbr^d$, $\delta'>0$ and $R'>0$.  If $\beta_{loc}^x>0$ and  $0<\lambda<\beta_{loc}^x$ then, for any $T>0$,
\[
V^\lambda(X_{\cdot},[0,T])=+\infty \quad \mathbb{P}^x\text{-a.s.}\,.
\]
\end{theorem}

\begin{proof} Choose any $\varepsilon\in(0,2)$, $R\in(0, R')$ and $T>0$. Consider the first exit time from the ball $B_{R}(x)$ of $X^x$,
$
\sigma=\sigma^x_{B_R(x)}:=\inf\{t\ge0: X_t^x\notin B_{R}(x)\}.
$
By Proposition~\ref{prop:boundinprob},
$
\mathbb{P}^x(\sigma<t)=\mathbb{P}^x((X_\cdot-x)^*_t\ge R)\le c_d t H(x,R).
$
Pick ${\bar{t}}=\varepsilon/(2c_dH(x,R))$ so that $\mathbb{P}^x(\sigma<{\bar{t}})\le\varepsilon/2$. Then for any integer $n\ge1$ and $\lambda>0$ we have
\[
\begin{split}
\mathbb{E}^x e^{-V^\lambda(X_\cdot,[0,T])}&\le \mathbb{E}^x e^{-V^\lambda(X_\cdot,[0,T\wedge \sigma])}\\
&=\mathbb{E}^x e^{-V^\lambda(X_\cdot^\sigma,[0,T])}\left(\mathbf{1}_{\{\sigma<{\bar{t}}\}}+\mathbf{1}_{\{\sigma\ge {\bar{t}}\}}\right)\\
&\le \mathbb{P}^x(\sigma<{\bar{t}})+\mathbb{E}^x\exp\left\{-\sum_{j=1}^n\left|X_{\frac{{\bar{t}}j}n}^\sigma-X_{\frac{{\bar{t}}(j-1)}n}^\sigma
\right|^\lambda\right\}\mathbf{1}_{\{\sigma\ge {\bar{t}}\}}\\
&\le \frac{\varepsilon}2+E.
\end{split}
\]
To simplify notation, for a fixed $y\in B_R(x)$, let $\tilde{\sigma}:=\sigma^y_{B_{2R}(y)}$. Then, since $\{\sigma\ge {\bar{t}}\}\subset\{\sigma^{y}_{B_{2R}(y)}\ge \bar{t}/n, y=X_{\frac{\bar{t}(n-1)}{n}}^{\sigma}\}$, using the strong Markov property, we get
\[\begin{split}
E&\le\mathbb{E}^x\left(\exp\left\{-\sum_{j=1}^{n-1}\left|X_{\frac{{\bar{t}}j}n}^{\sigma}-X_{\frac{{\bar{t}}(j-1)}n}^{\sigma}\right|^\lambda\right\}\mathbf{1}_{\{\sigma\ge \frac{{\bar{t}}(n-1)}{n}\}}\right.\\
&\qquad\qquad\qquad\times
\left.\mathbb{E}^x\left[\exp\left\{-\left|X_{{\bar{t}}}^{\sigma}-X_{\frac{{\bar{t}}(n-1)}{n}}^{\sigma}\right|^\lambda\right\}\mathbf{1}_{\{\sigma\ge {\bar{t}}\}}\Big|\mathcal{F}_{\frac{{\bar{t}}(n-1)}{n}\wedge\sigma}\right]\right)\\
&\le\mathbb{E}^x\left(\exp\left\{-\sum_{j=1}^{n-1}\left|X_{\frac{{\bar{t}}j}n}^{\sigma}-X_{\frac{{\bar{t}}(j-1)}n}^{\sigma}\right|^\lambda\right\}\mathbf{1}_{\{\sigma\ge \frac{{\bar{t}}(n-1)}{n}\}}\right.\\
&\qquad\qquad\qquad\times
\left.\mathbb{E}^y\left[\exp\left\{-\left|X_{{\bar{t}}\wedge\sigma-\frac{\bar{t}(n-1)}{n}\wedge\sigma}-y\right|^\lambda\right\}\mathbf{1}_{\{\tilde{\sigma}\ge {\bar{t}/n}\}}\right]\Big|_{y=X_{\frac{{\bar{t}}(n-1)}{n}}^{\sigma}}\right)\\
&\le\left(\sup_{y\in B_R(x)}\mathbb{E}^ye^{-|X_{\bar{t}/n}^{\tilde{\sigma}}-y|^\lambda}\right)\\
&\qquad\qquad\qquad\times\mathbb{E}^x\left(\exp\left\{-\sum_{j=1}^{n-1}\left|X_{\frac{{\bar{t}}j}n}^{\sigma}-X_{\frac{{\bar{t}}(j-1)}n}^{\sigma}\right|^\lambda\right\}
\mathbf{1}_{\{\sigma\ge \frac{{\bar{t}}(n-1)}{n}\}}\right)\\
&\le\cdots\le \left( \sup_{y\in B_R(x)}\mathbb{E}^ye^{-|X_{\bar{t}/n}^{\tilde{\sigma}}-y|^\lambda}\right)^n.
\end{split}
\]
Following \cite[p. 499]{blumenthalget61}, let
\[
A(t)=e^{-B(t)}:=\sup_{y\in B_R(x)}\mathbb{E}^ye^{-|X_{t}^{\tilde{\sigma}}-y|^\lambda},\qquad t>0, x\in\bbr^d, R>0,
\]
and consider
\[
\frac{1-A(t)}{t}=\inf_{y\in B_R(x)}\frac{1-\mathbb{E}^ye^{-|X_{t}^{\tilde{\sigma}}-y|^\lambda}}{t},
\]
where  for any $y\in B_R(x)$, using Fubini theorem and recalling Eq. \eqref{flambda},
\[
\begin{split}
\mathbb{E}^ye^{-|X_{t}^{\tilde{\sigma}}-y|^\lambda}&=\int_{\bbr^d}e^{-|v|^\lambda}\mathbb{P}^y(X_t^{\tilde{\sigma}}-y\in dv)\\
&=\int_{\bbr^d}f_\lambda(\xi)\int_{\bbr^d}e^{iv'\xi}\mathbb{P}^y(X_t^{\tilde{\sigma}}-y\in dv)d\xi\\
&=\int_{\bbr^d}f_\lambda(\xi)\mathbb{E}^ye^{i(X_t^{\tilde{\sigma}}-y)'\xi}d\xi\\
&=\int_{\bbr^d}f_\lambda(\xi)\Re\mathbb{E}^ye^{i(X_t^{\tilde{\sigma}}-y)'\xi}d\xi.
\end{split}
\]
Therefore, using Lemmas \ref{lem:infexchange} and \ref{lem:infiniteintegral}, for any $\kappa=R\in(0,(R'/2)\wedge\delta')$ with $\delta'>0$ from Condition (A$_x$),
\[
\begin{split}
\liminf_{t\downarrow0}\frac{1-A(t)}{t}&=\liminf_{t\downarrow0}\inf_{y\in B_R(x)}\int_{\bbr^d}f_\lambda(\xi)\frac{1-\Re\mathbb{E}^ye^{i(X_t^{\tilde{\sigma}}-y)'\xi}}{t}d\xi\\
&=\liminf_{t\downarrow0}\underline{h}(t,x)=\underline{D}(x)=+\infty\qquad \forall \lambda<\inf_{y\in B_R(x)}\beta^{y}.
\end{split}
\]

This implies that $B(t)/t\to+\infty$ as $t\downarrow0$ as well, so we can pick an integer $n_0\ge1$ such that $(n/{\bar{t}})B({\bar{t}}/n)\ge {\bar{t}}^{-1}\ln(2/\varepsilon)=2c_dH(x,R)(1/\varepsilon)\ln(2/\varepsilon)$ for any integer $n\ge n_0=n_0(\varepsilon, x,R)$.
Collecting the estimates, we obtain
\[
\mathbb{E}^x e^{-V^\lambda(X_\cdot,[0,T])}\le \frac{\varepsilon}{2}+A^n({\bar{t}}/n)=\frac{\varepsilon}{2}+e^{-nB({\bar{t}}/n)}\le\varepsilon,\quad \forall n\ge n_0.
\]
Letting $\varepsilon\downarrow0$, we obtain $V^\lambda(X_\cdot,[0,T])=+\infty$ $\mathbb{P}^x$-a.s. for any $\lambda<\inf_{y\in B_R(x)}\beta^{y}$, which implies the claim of the theorem since we can choose a small enough $R$ so that $\inf_{y\in B_R(x)}\beta^{y}$ (which is nondecreasing as $R\downarrow0$) is arbitrarily close to $\beta_{loc}^x$.
\end{proof}

\begin{example}\label{example:2} Let $Y_t^x$ be the process of Example \ref{example:1}. Then, by Theorem \ref{thm:firstcrit}, $V^p(Y_\cdot^x,[0,T])<+\infty$ $\mathbb{P}^x$-a.s. for any $x\in\bbr$, $T>0$ and $p>2=\beta_\infty^{unif}$. On the other hand, Condition (A$_x$) holds for any $x\in\bbr^d$, so by Theorem \ref{thm:lowerbound}, $V^p(Y_\cdot^x,[0,T])=+\infty$ $\mathbb{P}^x$-a.s. for any $x\in\bbr\setminus\{0\}$, $T>0$ and $p<2=\beta_{loc}^{x}$. Of course, $V^p(Y_\cdot^0,[0,T])<+\infty$ $\mathbb{P}^0$-a.s. for any $p>0$ and $T>0$ as $Y_\cdot^0\equiv0$. The critical case $p=2$ is not covered by our theorems, but can be deduced from the unboundedness of the $2$-variation of Brownian motion and the fact that $x\mapsto e^x$ is locally Lipschitz, that is, if $p=2$, $V^p(Y_\cdot^x,[0,T])=+\infty$  $\mathbb{P}^x$-a.s. for any $x\ne0$.
\end{example}

\section{Generalizations and Applications}\label{ch:applications}

In this section, we consider various examples and generalize our results from Sect. 2 to Hunt semimartingales.

The following simple observation will become handy: Since all norms on $\bbr^d$ are equivalent, to check the finiteness of $V^p(g; [a,b])$ one can use any norm, in particular, by taking the $L_1$-norm, one can easily verify that the finiteness of $V^p(g; [a,b])$ is equivalent to the finiteness of $V^p(g_i; [a,b])$ for all $i=1,\dots,d$ where $g=(g_1,\dots,g_d)$. Indeed, if $c$ and $C$ are positive constants such that $c\norm{\cdot}_{L_1}\le \norm{\cdot}\le C\norm{\cdot}_{L_1}$ on $\bbr^d$, then $c^p(\mathbf{1}_{\{p\ge1\}}+d^{p-1}\mathbf{1}_{\{0<p<1\}})\max_{i=1,\dots,d}V^p(g_i; [a,b])\le V^p(g; [a,b])\le C^p(d\vee d^{p})\max_{i=1,\dots,d}V^p(g_i; [a,b])$.
 Thus, the $p$-variation of $g$ is determined by the worst behaving component $g_j$. By defining the $p$-variation index of $g$ as
\[
v(g;[a,b]):=\begin{cases} \inf\{p>0\,:\, V^p(g; [a,b])<+\infty\} &\text{if the set  is nonempty,}\\
+\infty & \text{otherwise},
\end{cases}
\]
one has $v(g;[a,b])=\max\{v(g_1;[a,b]),\dots,v(g_d;[a,b])\}$.


\subsection{L\'evy Driven SDEs}\label{sec:LevySDE}

Let us consider the following stochastic differential equation driven by an $\bbr^m$-valued L\'evy process $(Z_t)_{t\geq 0}$ for every starting point $x\in\bbr^d$,
\begin{align} \begin{split} \label{levysde}
  dX_t^x&=\Phi(X_{t-}^x) \,dZ_t, \\
  X_0^x&=x
\end{split} \end{align}
where $\Phi: \bbr^d \to \bbr^{d \times m}$ is \emph{admissible}, that is, the equation admits a unique conservative solution. Having $\Phi$ locally Lipschitz continuous and satisfying the standard linear growth condition would be sufficient. If $\Phi$ is locally bounded and finely continuous, the solution $X$ of this stochastic differential equation admits the symbol
\begin{equation}\label{symbolSDE}
    q(x,\xi)=\psi(\Phi(x)'\xi).
\end{equation}
where $\psi:\bbr^m\to\bbc$ denotes the characteristic exponent of the L\'evy process.

This was shown in \cite[Theorem 3.1]{sdesymbol}. In that article, properties of the process could only be obtained for the case of bounded $\Phi$ because in general the solution of the above SDE is not rich Feller. By \cite[Theorem 3.33]{cinlarjacod81},  it is easy to see that the solution is a L\'evy-type process. Furthermore, since $\Phi$ is finely continuous and $\psi$ is continuous, the symbol and the differential characteristics are finely continuous, too.

For the remainder of the section we assume that $\Phi$ is admissible, finely continuous and locally bounded. Lipschitz continuity would be sufficient to fulfill all three requirements. First, we prove the following lemma which is closely connected with \cite[Theorem 5.7]{sdesymbol}
(there $\Phi$ can be unbounded, but $d=n$ and $\xi\mapsto \Phi(x)'\xi$ is a bijection for all $x\in\bbr^d$):

\begin{lemma} \label{lem:indextransfer}
Consider the SDE \eqref{levysde}. If $\psi$ is bounded or if $\Phi$ is bounded and there exists an $x\in\bbr^d$
such that the mapping $\xi\mapsto\Phi(x)'\xi$ is surjective (i.e. onto) then
\[
\beta_\infty^{unif,1}(X)=\beta_\infty^{unif}(X)=\beta_1(Z).
\]
\end{lemma}

\begin{proof}
If the characteristic exponent of $Z$ is bounded (e.g. if $Z$ is a simple Poisson process), then the symbol $q(x,\xi)$ of $X$ is also uniformly bounded and so all three
indices are 0. So consider the other case.

Let $x_0$ be a point such that $\xi_0\mapsto\Phi(x_0)'\xi$ is surjective and for any set $A$ containing $x_0$ define
\[
 S(A):=\sup_{x\in A}\norm{\Phi(x)'}_{op}
\]
where $\norm{\Phi(x)'}_{op}$ denotes the operator norm of $\xi\mapsto\Phi(x)'\xi$. Set $S=S(\bbr^d)$.
Since $\xi\mapsto\Phi(x_0)'\xi$ is a linear surjective map, by the open mapping theorem,
we obtain that the set $\{\Phi(x_0)'\varepsilon: \norm{\varepsilon}< 1 \}$ is open in $\bbr^m$ and contains the origin, and so there exists an $r_0>0$ such that for any $r\in]0,r_0]$ we obtain the following chain of inclusions:
\begin{align*}
\{\delta\in\bbr^m:\norm{\delta}\leq r\} \subseteq \{\Phi(x_0)'\varepsilon: \norm{\varepsilon}< 1 \}
&\subseteq \{\Phi(y)'\varepsilon: y\in A,\norm{\varepsilon}\leq 1 \} \\ &\subseteq \{\delta \in \bbr^m: \norm{\delta}\leq S(A) \}.
\end{align*}
Therefore, with $A=\bbr^d$, we have
\[
H_\psi\left(\frac{R}{r}\right)=\sup_{\norm{\delta}\le r}\left|\psi\left(\frac{\delta}{R}\right)\right| \leq H_q(R) \leq \sup_{\norm{\delta}\le S}\left|\psi\left(\frac{\delta}{R}\right)\right|=H_\psi\left(\frac{R}{S}\right).
\]
Since for every $\lambda > 0$ such that $\limsup_{R\to 0} R^\lambda H(R) =0$,
we obtain  that $\limsup_{R\to 0} R^\lambda H(c\cdot R) =0$ for any $c>0$, this proves $\beta_\infty^{unif}(X)=\beta_\infty^{unif}(Z)=\beta_1(Z)$
by the definition of $\beta_\infty^{unif}$.

By taking $A=B_{2R}(x_0)$, we similarly obtain
\[
H_\psi\left(\frac{R}{r}\right)=\sup_{\norm{\delta}\le r}\left|\psi\left(\frac{\delta}{R}\right)\right| \leq H_q(x_0,R)
\leq \sup_{\norm{\delta}\le S(B_{2R}(x_0))}\left|\psi\left(\frac{\delta}{R}\right)\right|=H_\psi\left(\frac{R}{S(B_{2R}(x_0))}\right).
\]
Using the first inequality, the definitions of $\beta^{x_0}_\infty$ and $\beta_\infty^{unif,1}$, 
 we get
 \[\beta_1(Z)=\beta_\infty^{unif}(Z)\le \beta_{\infty}^{x_0}(X)\le
\sup_{x\in\bbr^d}\beta_\infty^x=\beta_\infty^{unif,1}(X)\le \beta_\infty^{unif}(X),\]
completing the proof.
\end{proof}

\begin{remark}
 If $d=m$, then the continuous mapping $\xi\mapsto\Phi(x_0)'\xi$ being surjective is also injective,
 hence possesses an inverse in which case $r_0$ in the proof can be taken as $r_0=\norm{(\Phi(x_0)')^{-1}}_{op}^{-1}$.
 \end{remark}
\begin{example} The assumption concerning surjectivity (or the injectivity of the mapping $\zeta\mapsto\Phi(x_0)\zeta$) cannot be dropped.
Indeed, let $W$ be a standard Brownian motion and $Z$ be the symmetric $\alpha$-stable L\'evy motion, independent of $W$, where $\alpha\in]0,2[$. The trivariate L\'evy process $(t, Z_t, W_t)'_{t\geq 0}$ has the index $\beta_\infty^{unif}=2$. Now consider
\[
  dX_t^x=\left( \begin{array}{ccc} 1 & 0 & 0\\ 0 & 1 & 0\end{array} \right) \,d\begin{pmatrix}
                                                                                              t \\
                                                                                              Z_t \\
                                                                                              W_t \\
                                                                                            \end{pmatrix}
  , \quad X_0^x=x=(x_1,x_2,x_3)'.
\]
The solution is $(t+x_1, Z_t+x_2)'_{t\ge0}$ with index $\beta_\infty^{unif}(X)=\max\{1,\alpha\}$.
\end{example}

If $\Phi$ is unbounded, the index $\beta_\infty^{unif}(X)$ need not exist, since we could have $H(R)=\infty$ for all $R>0$.
However, we have the following proposition
\begin{proposition}\label{prop:unboundedcase}
Let $X$ be the solution of \eqref{levysde}, where for $\Phi$ there exists an $x_0\in\bbr^d$ such that the mapping $\xi\mapsto\Phi(x_0)'\xi$ is surjective.
Assume that both the characteristic exponent $\psi$ of the driving L\'evy process $Z$ and the function $\Phi$ are unbounded,
but $\beta_\infty^{unif}(X)<+\infty$. Then
\[
\beta_\infty^{unif,1}(X)=\beta_\infty^{unif}(X)=\beta_1(Z).
\]
\end{proposition}
\begin{proof}
Since in the proof of Lemma \ref{lem:indextransfer}, the boundedness of $\Phi$ was not used to get $\beta_1(Z)\le \beta_\infty^{unif,1}$, we only need to
show that $\beta_1(Z)=\beta_\infty^{unif}(X)$ whenever the latter index is finite and $X$ is a solution of \eqref{levysde}. Consider the set of directions
along which the characteristic exponent $\psi$ of $Z$ remains bounded, i.e.
\[
BD_\psi:=\left\{y\in\bbr^m: \sup_{r\ge0}|\psi(ry)|<+\infty\right\}.
\]
Since $\psi$ is negative-definite and so $\sqrt{|\psi(\cdot)|}$ is subadditive (c.f. \cite{bergforst}), we get that $BD_\psi$ is a linear subspace of $\bbr^m$.
Indeed, if $y_1,y_2\in BD_\psi$ and $c_1,c_2\in\bbr$, by considering finite positive constants $M_i, i=1,2$ such that
\[
\sup_{r\ge0}|\psi(ry_i)|\le M_i,\quad i=1,2,
\]
we get for any $r\ge0$
\[
\begin{split}
\sqrt{|\psi(r(c_1y_1+c_2y_2))|}&\le \sqrt{|\psi(rc_1y_1)|}+\sqrt{|\psi(rc_2y_2)|}\\
&=\sqrt{|\psi(r|c_1|y_1)|}+\sqrt{|\psi(r|c_2|y_2)|}\le\sqrt{M_1}+\sqrt{M_2}
\end{split}
\]
so that $c_1y_1+c_2y_2\in BD_\psi$. If $BD_\psi\ne\{0\}$, let $\{y_i, i=1,\dots,k\}$ form a basis of $BD_\psi$ where $0<k=\mathrm{dim}\,BD_\psi\le m$.
Then for any $y\in BD_\psi$ we have $y=\sum_{i=1}^ka_iy_i$ for some coefficients $a_i$ and, by the subadditivity of $\sqrt{|\psi(\cdot)|}$ again,
\[
\sqrt{|\psi(y)|}\le\sum_{i=1}^k\sqrt{|\psi(|a_i|y_i)|}\le\sum_{i=1}^k\sqrt{M_i},
\]
for some positive and finite constants $M_i, i=1,\dots,k$. This shows that $|\psi(\cdot)|$ is uniformly bounded on $BD_\psi$, i.e.
\[
\sup_{y\in BD_\psi}|\psi(y)|\le M:=\left\{\sum_{i=1}^k\sqrt{M_i}\right\}^2.
\]
Since $\psi$ is unbounded and
\[
\bbr^m=BD_\psi\bigoplus BD_\psi^\bot,
\]
the orthogonal complement $BD_\psi^\bot$ contains a nonzero vector and so $k<m$. Then for any $x\in\bbr^d$ and $\varepsilon\in B_1(0)\subset\bbr^d$ we have
\[
\Phi(x)'\varepsilon=\mathrm{proj}_{BD_\psi}\Phi(x)'\varepsilon+\mathrm{proj}_{BD_\psi^\bot}\Phi(x)'\varepsilon=:y_1(x,\varepsilon)+y_2(x,\varepsilon).
\]
Then using the subadditivity of $\sqrt{|\psi(\cdot)|}$ again,
\[
\sqrt{|\psi(y_2(x,\varepsilon)/R)|}\le \sqrt{|\psi(\Phi(x)'\varepsilon/R)|}+\sqrt{|\psi(y_1(x,\varepsilon)/R)|}\le \sqrt{H(R)}+\sqrt{M}<+\infty
\]
for all $R$ small enough for which  $H(R)<+\infty$ since $\beta_\infty^{unif}(X)<+\infty$. This yields
\begin{equation}\label{eq:bddprojections}
\sup_{x\in\bbr^d}\sup_{\norm{\varepsilon}\le 1}|\psi(\mathrm{proj}_{BD_\psi^\bot}\Phi(x)'\varepsilon/R)|< \infty,
\end{equation}
for all $R$ small enough. But this means that
\[
K:=\sup_{x\in\bbr^d}\sup_{\norm{\varepsilon}\le 1}\norm{\mathrm{proj}_{BD_\psi^\bot}\Phi(x)'\varepsilon}< \infty.
\]
Indeed, if there was a sequence $\{(x_n, \varepsilon_n)\}_{n=1}^\infty\subset \bbr^d\times B_1(0)$ such that $\norm{y_2(x_n,\varepsilon_n)}\to\infty$
as $n\to\infty$, then by considering any orthonormal basis $\{y_j, j=k+1,\dots,m\}$ of $BD_\psi^\bot$ we would get at least one sequence of coefficients $\{a_{j,n}\}_{n=1}^\infty$,
$a_{j,n}:=y_2(x_n,\varepsilon_n)'y_j, j=k+1,\dots,m$ diverging to infinity, say for $j=j_0$. But then $|\psi(a_{j_0,n}y_{j_0})|\to\infty$ as $n\to\infty$,
contradicting  \eqref{eq:bddprojections}.

So for any $x\in\bbr^d$ and $\varepsilon\in B_1(0)\subset\bbr^d$, and any $\lambda>\beta_1(Z)$, we can write
\[\begin{split}
\sqrt{|\psi(\Phi(x)'\varepsilon/R)|}&\le \sqrt{|\psi(y_1(x,\varepsilon)/R)|}+\sqrt{|\psi(y_2(x,\varepsilon)/R)|}\\
&\le \sqrt{M}+\sqrt{C|(y_2(x,\varepsilon)/R)|^\lambda}\\
&\le \sqrt{M}+\sqrt{C(K/R)^\lambda},
\end{split}
\]
for some $C>0$ which can only depend on $\lambda$ and all $R$ small enough. This implies
\[
H(R)\le \tilde{C} R^{-\lambda}
\]
for some finite positive constant $\tilde{C}=\tilde{C}(\lambda)$ and all $R$ small enough, yielding $\beta_\infty^{unif}(X)<\tilde{\lambda}$ for any $\tilde{\lambda}>\lambda>\beta_1(Z)$ and finishing the proof.
\end{proof}
To illustrate the just proved proposition, consider the following example.
\begin{example} Let $\{W_t, t\ge0\}$ be a standard one-dimensional Brownian motion and let $\{N_t, t\ge0\}$ be a simple one-dimensional Poisson process with intensity $\lambda$, independent of $\{W_t\}$. Consider the following two-dimensional SDE, driven by the L\'evy process $\{Z_t=(W_t, N_t)', t\ge0\}$:
\[
dX_t^x=\left(
         \begin{array}{cc}
           1 & 0 \\
           0 & X_{t-,2}^{x} \\
         \end{array}
       \right)dZ_t,\quad X_0^x=x=\left(
                                       \begin{array}{c}
                                         x_{0,1} \\
                                         x_{0,2} \\
                                       \end{array}
                                     \right)
\]
where $X_t^x=(X_{t,1}^x, X_{t,2}^x)'$. Then
\[
q(x,\xi)=\psi(\Phi(x)'\xi)=\frac12\xi_1^2+\lambda\left(1-e^{ix_2\xi_2}\right)
\quad \textrm{for}\ \ x=\left(
                                                                                         \begin{array}{c}
                                                                                           x_1 \\
                                                                                           x_2 \\
                                                                                         \end{array}
                                                                                       \right),
                                                                                       \xi=\left(
                                                                                         \begin{array}{c}
                                                                                           \xi_1 \\
                                                                                           \xi_2 \\
                                                                                         \end{array}
                                                                                       \right).
\]
The obtained symbol is clearly bounded in $x$ and grows only along the $\xi_1$ direction. Of course, $\beta_\infty^{unif}(X)=2=\beta_1(Z)=\beta_2(Z)$.
\end{example}
Despite seemingly restrictive condition $\beta_\infty^{unif}(X)<+\infty$, using a stopping technique and a sequence of SDEs we are still able to derive a very applicable result on the $p$-variation of the solution $X$ of the SDE \eqref{levysde}.

\begin{theorem} \label{thm:levysde}
Let $X$ be the solution of \eqref{levysde}, where $\Phi$ is locally Lipschitz and satisfies the standard linear growth condition. Assume furthermore that there exists an $x_0\in\bbr^d$ such that the mapping $\xi\mapsto\Phi(x_0)'\xi$ is surjective.
Then for every $T>0$, the $p$-variation of $X$ on $[0,T]$ is finite $\bbp^x$-almost surely for $p>\beta_1(Z)$, the classical upper Blumenthal--Getoor index of the driving L\'evy term $Z$, or in terms of the $p$-variation index, $v(X, [0,T])\le v(Z, [0,T])$. {Moreover, $v(X^x, [0,T])\ge \beta_2(Z)$, $\mathbb{P}^x$-a.s. whenever $\Phi(x)\ne0$ and where $\beta_2$ is given by \eqref{beta2}.}
\end{theorem}

\begin{remark} For L\'evy processes, due to the results of Blumenthal and Getoor \cite{blumenthalget61}, Monroe \cite{Monroe} and also Bretagnolle \cite{Bretagnolle}, one has $\beta_1(Z)=v(Z;[0,T])$, for any $T>0$, so Theorem \ref{thm:levysde} simply means that the solution of \eqref{levysde} cannot be worse than the driving process $Z$ in terms of $p$-variation. Moreover, if $\beta_1(Z)=\beta_2(Z)$ (which happens if the drift of $Z$ has no dominating effect; see Remark \ref{rem:indices_Levy case}), then the index of $p$-variation of the solution is the same as that of the driving process $Z$.
\end{remark}

\begin{proof}
We may assume w.l.o.g. that $\xi\mapsto\Phi(0)'\xi$ is surjective and define a sequence of stopping times as follows:
\[
\tau_n:=\inf\{t\geq 0 : X_t \notin B_n(0) \}.
\]
For every $n\in\bbn$, the stopped process $X^{\tau_n}$ coincides with the solution $X^n$ of the SDE
\begin{align*}
  dX^n_t&=\Phi^n(X_{t-}^n) \,dZ_t, \\
  X_0^n&=x
\end{align*}
on the stochastic interval $\rosi{0}{\tau_n}$, $\bbp^x$-a.s. for every starting point $x\in\bbr^d$. Here, $\Phi^n:=\Phi\cdot \chi_n$ where $\chi_n\in C_c^\infty$ is a smooth cut-off function such that
\[
1_{B_n(0)} \leq \chi_n \leq 1_{B_{n+1}(0)}.
\]
It is easily seen that for every $n\in\bbn$ the coefficient $\Phi^n$ is again locally Lipschitz and retains at most linear growth. Moreover, it is now bounded. Furthermore, $\xi\mapsto\Phi^n(0)'\xi$ is surjective, and therefore, we are in the setting of Lemma \ref{lem:indextransfer}. Hence $\beta_\infty^{unif}(X^n)=\beta_\infty^{unif}(Z)=\beta_1(Z)$.  Since the process $X$ is conservative, we have $\tau_n\to\infty$ for $n\to\infty$, and hence for every $x\in\bbr^d$ and $\bbp^x$-almost every $\omega\in\Omega$ there exists an $n_0\in\bbn$ such that $\tau_{n_0}(\omega) > T$. As $X^{\tau_{n_0}}$ agrees with $X$ on $[0,T]$ and, by Theorem~\ref{thm:levysde}, $V^p(X^{n_0},[0,T])<\infty$ for $p>\beta_1(Z)$, this yields the first claim.

To get the second claim, first notice that
\[
\beta^x_{loc}(X)=\left\{
                   \begin{array}{ll}
                     \beta_2(Z), & \hbox{if $\Phi(x)\ne0$;} \\
                     0, & \hbox{if $\Phi(x)=0$.}
                   \end{array}
                 \right.
\]
since if $\Phi(x)>0$ (resp., $\Phi(x)<0$) then $\Phi(y)>0$ (resp., $\Phi(y)<0$) for all $y$ in a small ball around $x$, $B_r(x)$, by continuity of $\Phi$, and for such $y\in B_r(x)$, $\beta^y(X)=\beta_2(Z)$ (see \eqref{spot beta} and \eqref{symbolSDE}).

Furthermore, since $q(x,\xi)=\psi(\Phi(x)'\xi)=-\ln\mathbb{E}\exp\{i\xi'(\Phi(x)Z_1)\}$, if $\Phi(x)\ne0$, we have $\Phi(y)\ne0$ for all $y\in B_\kappa(x)$ for some $\kappa>0$, and the differential characteristics of $X$ satisfy the conditions of Corollary \ref{cor:sufficient_Ax}. Indeed, if the L\'evy triplet (with cut-off function $\chi_Z$) of $Z$ is $(\ell_Z, Q_Z, N_Z(\cdot))$, then
\[
\begin{split}
&\ell_X(x)=\Phi(x)\ell_Z, \quad Q_X(x)=\Phi(x)Q_Z\Phi(x)',\quad\text{and}\\
&  V_3(x,\xi)=\int_{\{z\,:\, \Phi(x)z\ne0\}} \left(e^{i\xi'\Phi(x)z}-1-i\xi'\Phi(x)z\chi_Z(z)\right)N_Z(dz).
\end{split}
\]
Therefore, Theorem \ref{thm:lowerbound} is applicable since, with $\Phi$ locally Lipschitz continuous by assumption, $\ell_X$ and $Q_X$ are clearly continuous, and
the continuity of $x\mapsto V_3(x,\xi)$ for any fixed $\xi$  follows by dominated convergence since $v\mapsto e^{i\xi'\Phi(v)z}-1-i\xi'\Phi(v)z\chi_Z(z)$ is continuous and bounded by an $N_Z(\cdot)$-integrable function $C_{\kappa,\Delta} (1\wedge \norm{z}^2)$, where the positive constant $C_{\kappa,\Delta}$ does not depend on $(v,\xi)\in B_\kappa(x)\times B_\Delta(0)$ for $\kappa>0$ as above and any fixed $\Delta>0$, yielding the second claim.
\end{proof}

\begin{remark} \label{rem:finance}
The above theorem holds true in more generality: in fact, it is enough to require that $\Phi$ as well as $\Phi^n$ (for every $n\in\bbn$) are admissible. In this version we will use the theorem in the context of processes used in mathematical finance (see below).
\end{remark}

\subsection{Hunt Semimartingales}

It is a natural question, whether the symbol can be further generalized. A natural class for such a generalization would be Hunt semimartingales. We show in the following example that this is not possible. Thus giving a partial answer to a question raised in \cite{levymatters3}.

\begin{example}
The Cantor process $X$ was constructed in \cite[Sect.~4]{detmp1}. Here we only need to know that there exists a process which is a Hunt semimartingale, but not a L\'evy-type process and which has the following properties: The process $X$ is deterministic and starting at zero in zero we have
\[
X_t^0=f(t)=\frac{1}{2}(c(t)+t) \text{ on } [0,1]
\]
where $c$ is the Cantor function. If we tried to calculate the symbol for $x=0$ and $\xi=1$, we would have to consider the limit as $t\downarrow 0$ of
\[
\frac{e^{if(t)}-1}{t} = \frac{\cos(f(t))-1}{t} + i\frac{\sin(f(t))}{t}.
\]
Now we plug the sequence $t_n=3^{-n}$ (which tends to zero) into the imaginary part and obtain:
\[
\frac{\sin(f(t_n))}{t_n}=\frac{f(t_n)}{t_n} \cdot \frac{\sin(f(t_n))}{f(t_n)} =\frac{1}{2} \Big( \frac{2^{-n}}{3^{-n}}+1\Big) \cdot \frac{\sin(f(t_n))}{f(t_n)}
\]
This tends to infinity and hence the limit \eqref{stoppedsymbol} does not exist. Using a stopping time does not work in this deterministic setting.
\end{example}

However, we can use a different approach in order to bring our methods into account: by a result due to E. Cinlar and J. Jacod (see \cite[Theorem 3.35]{cinlarjacod81}) every Hunt semimartingale can be written as a random time change of a L\'evy-type process, that is, for every Hunt semimartingale $Y$ there exists a strictly monotone increasing continuous process $\{A_t, t\ge0\}$ (a subordinator) such that
\[
Y_t=X_{A_t}
\]
where $X$ is a L\'evy-type process. We call $X$ a parent process of $Y$ and obtain the following result for the subordinated process $Y$:

\begin{theorem}
Let $Y$ be a Hunt semimartingale and $X$ a parent process of $Y$ with symbol $q$ and index $\beta_\infty^{unif}(X)<+\infty$. For for every $T>0$ and $p > \beta_\infty^{unif}(X)$ we obtain
\[
  V^p(Y^y;[0,T]) < \infty\quad\text{$\bbp^y$-almost surely.}
\]

\end{theorem}

\subsection{Stable-like Processes}

This class of processes which has been studied by R. Bass \cite{bass88a} and A. Negoro \cite{negoro94}, among others, got  back into the focus of interest recently. N.~Sandric \cite{sandric} has studied the long-time behavior of processes in this class, R.L. Schilling and J. Wang \cite{schillingwang} have dealt with transience and local times.

A Feller process $X$ on $\bbr$ with symbol of the form $q(x,\xi)=\norm{\xi}^{2a(x)}$ is called a \emph{stable-like process}.
As it is common in the literature, we assume that the function $a$ is uniformly bounded, that is, $0<a_0\leq a(x) \leq a_\infty <1$. We obtain the following (for $R\in]0,1]$):
\[
H(R)=\sup_{y\in\bbr^d} \sup_{\norm{\varepsilon}\leq 1} \abs{\frac{\norm{\varepsilon}^{2a(y)}}{R^{2a(y)}}}= \sup_{y\in\bbr^d} \abs{\frac{1}{R^{2a(y)}}}=\frac{1}{R^{2a_\infty}}.
\]
Hence, $\beta_\infty^{unif}(X)=2a_\infty$, and therefore, by our Theorem \ref{thm:firstcrit}: For every $T>0$ and $p>2a_\infty$, we obtain
\[
  V^p(X^x;[0,T]) < \infty\quad \bbp^x\text{-almost surely}.
\]
In a similar way, we get $\beta^y(X)=2a(y)$ and so $\beta^x_{loc}(X)=2\sup_{R>0}\inf_{y\in B_R(x)}a(y)$. In particular, if $x\mapsto a(x)$ is continuous, then $\beta^x_{loc}(X)=2a(x)$ and
\[
  V^p(X^x;[0,T]) = \infty\quad \bbp^x\text{-almost surely}
\]
for any $p<2a(x)$. This is due to the fact that the continuity of $a$ is inherited by the characteristics.

\subsection{Stochastic Exponential of a L\'evy Process}\label{sec:stoch-exp}
Arguably the simplest and most important case of \eqref{levysde} is when $d=1$, $\Phi(y)=y$. The unique strong solution is then given by $X_t^x=x\mathcal{E}(Z)_t$ where $\mathcal{E}(Z)$ is the Dol\'eans-Dade exponential of $Z$ (the solution of \eqref{levysde} with $x=1$; see, e.g., \cite{Maisonneuve} for the case when $Z$ is a local martingale with continuous paths,  \cite{Doleans-Dade} for the case when $Z$ is a semimartingale with right continuous paths, and \cite[Theorem 5.1]{Galchuk} when $Z$ is a semimartingale with regulated sample paths (i.e. having limits from the right and from the left)) given by
\begin{equation}\label{stoch_exp}
\mathcal{E}(Z)_t=\exp\left\{Z_t-\frac12\langle Z,Z\rangle_t^c\right\}\prod_{0<s\le t}(1+\Delta Z_s)e^{-\Delta Z_s}.
\end{equation}
Then Theorem \ref{thm:levysde} gives
\begin{corollary} If $Z$ is a L\'evy process, then, for any $x\in\bbr\setminus\{0\}$ and $T>0$, $v(x\mathcal{E}(Z); [0,T])\in[\beta_2(Z), \beta_1(Z)]$ $\mathbb{P}^x$-almost surely.
\end{corollary}
\begin{remark} One can also consider  equation \eqref{levysde},  when $d=1$, $\Phi(y)=y$ and $v(Z, [0,T])=p\in(0,2)$, path-by-path by using the so-called Young integrals (see \cite[p. 195]{DudNor}) and show that the solution inherits the $p$-variation index from the driving process precisely (see \cite[Thm. 5.21]{DudNor}). In the stochastic setting, we have showed the same result, but without the restriction $p\ne2$ if $Z$ has no dominating drift so that $\beta_2(Z)=\beta_1(Z)$.
\end{remark}

\subsection{Generalized Ornstein--Uhlenbeck Process}\label{OU_subsection}

The Ornstein--Uhlenbeck (OU) process and its various extensions are important in many areas. A possible starting point to extensive available literature on OU and generalized processes is a survey by Maller, M\"uller and Szimayer \cite{MMS} which we will follow here.

Let $Y_t=(U_t, L_t), t\ge0$ be a bivariate L\'evy process such that the L\'evy measure of $U_t$ is supported on $(-1,+\infty)$, i.e., $U_t$ has no jumps in $(-\infty,-1]$. Then the Dol\'eans-Dade stochastic exponential of $U$, namely $\mathcal{E}(U)_t=Z_t$ given by \eqref{stoch_exp}, is positive. Thus we can consider $\xi_t=-\log\mathcal{E}(U)_t$. As in \cite[Eq. (16)]{MMS}, we define
 a L\'evy process $\eta_t$  by
\[
\eta_t:=L_t-\sum_{0<s\le t}(1-e^{\Delta\xi_s})\Delta L_s +t\mathrm{Cov}(B_{\xi,1},B_{L,1}).
\]
Here $B_{\xi,1}$ and $B_{L,1}$ denote the Brownian components at time $t=1$ of $\xi$ and $L$, respectively.
With these definitions, the process $X_t$, called the \emph{generalized OU process} (corresponding to $(U_t, L_t)$) and given by
\[
X_t:=m(1-e^{-\xi_t})+e^{-\xi_t}\int_0^te^{-\xi_{s-}}d\eta_s + X_0e^{-\xi_t}, \quad t\ge0,
\]
where $m$ is a real constant,
$X_0$ is assumed to be $\mathcal{F}_0$-measurable (the filtration $(\mathcal{F}_t)_{t\ge0}$ is taken to be the natural one of $(\xi_t,\eta_t)_{t\ge0}$) and independent of $(\xi,\eta)_{t\ge0}$,
is the unique (up to indistinguishability) solution to
\[
dX_t=(X_{t-}-m)dU_t+dL_t, \quad t\ge0.
\]
Theorem \ref{thm:levysde} now yields
\begin{corollary}\label{genOUvar} Let $X_t^x, t\ge0$ be the generalized OU process corresponding to the bivariate L\'evy process $(U_t, L_t), t\ge0$ with $X_0=x$. Then for any $T>0$
\[v(X^x,[0,T])\in [\max\{\beta_2(U),\beta_2(L)\},\max\{\beta_1(U),\beta_1(L)\}] \quad \mathbb{P}^{x}-\text{a.s.} .
\]
\end{corollary}

\begin{proof}
We add one more equation and consider
\[
d\tilde{X}_t^x:=
d\left(
       \begin{array}{c}
        U \\
        X \\
       \end{array}
  \right)_t=
  \left(
        \begin{array}{cc}
        1 & 0 \\
        X_{t-}^x-m & 1 \\
        \end{array}
  \right)
 d \left(
  \begin{array}{c}
    U \\
    L \\
  \end{array}
  \right)_t,
  \quad \tilde{X}^x_0=\left(
  \begin{array}{c}
    U \\
    X \\
  \end{array}
  \right)_0
  =\left(
  \begin{array}{c}
    0 \\
    x \\
  \end{array}
  \right).
\]
Then, clearly,  we do have a unique solution of the augmented system,
 and Theorem \ref{thm:levysde} yields
\[
\begin{split}
v(\tilde{X}^x,[0,T])&=\max\{v(U,[0,T]),v(X^x,[0,T])\}\\
&\le \beta_1((U,L)')\\
&=v((U,L)',[0,T])\\
&=\max\{v(U,[0,T]),v(L,[0,T])\}\\
&=\max\{\beta_1(U),\beta_1(L)\}\qquad \mathbb{P}^x\textrm{-a.s. for any $T>0$.}
\end{split}
\]
Moreover, by the same theorem,
\[
v(\tilde{X}^x,[0,T])\ge \beta_2((U,L)')
=\max\{\beta_2(U),\beta_2(L)\}\qquad \mathbb{P}^x\textrm{-a.s. for any $T>0$.}
\]

Hence the claim.
\end{proof}
So in terms of the $p$-variation, the generalized OU process $X$ behaves no worse than the driving bivariate L\'evy process $Y$. A particular case of such a situation is when $U_t=-\gamma t$ and $L_t=\sigma W_t$, $t\ge0$, $\sigma\ge0$, $\gamma\in\mathbb{R}$ where $W_t, t\ge0$ is a standard Brownian motion, in other words, $X^x$ is the one-dimensional Gaussian OU process (if $\sigma\ne0$) which solves the SDE
\[
dX_t=\gamma(m-X_t)dt+\sigma dW_t, \quad t>0; X_0=x.
\]
Hence, by Corollary \ref{genOUvar}, $v(X^x,[0,T])\in[\max\{0\cdot\mathbf{1}_{\{\gamma\ne0\}}, 2\cdot\mathbf{1}_{\{\sigma\ne0\}}\}, \max\{\mathbf{1}_{\{\gamma\ne0\}}, 2\cdot\mathbf{1}_{\{\sigma\ne0\}}\}]$ $\mathbb{P}^x$-a.s. for any $T>0$. On the other hand, it is well-known (see, e.g. \cite[Eq. (3)]{MMS}) that $X^x$ can be represented as
\[
X_t^x=m(1-e^{-\gamma t})+\sigma e^{-\gamma t}W_{(e^{2\gamma t}-1)/(2\gamma)}+xe^{-\gamma t},\quad t\ge0,
\]
so if $\sigma\ne0$ then $X^x$ inherits the $p$-variation index from the Brownian motion (that is, $v(X^x,[0,T])=v(W,[0,T])=2$). If $\sigma=0$ and $\gamma\ne0$, $X^x$ is a non-random and nonconstant  continuous function of bounded 1-variation on any interval $[0,T]$. So the upper bound in Corollary~\ref{genOUvar} cannot be improved, in general, whereas the lower bound being zero is clearly worse in this case.
\subsection{Processes Used in Mathematical Finance}

Among many models used in mathematical finance, two stochastic volatility (SV) models (for an account on the history of such models, see, e.g. \cite{Shephard2008}) are particularly relevant to our investigation. The first is the OU-type process, introduced by O.~E.~Barndorff-Nielsen and N.~Shephard \cite{BNandS2001} and the other is the COGARCH model proposed by C.~Kl\"uppelberg, A.~Lindner and R.~Maller \cite{cogarch}. Both processes fit well into the considered framework of L\'evy-type processes and are solutions to SDEs driven by L\'evy processes of the type given in \eqref{levysde}. Here we investigate $p$-variation properties of these processes. The volatility of the respective processes is restricted to stay positive. By Remark \ref{rem:finance} it poses no problem to prolong $\Phi$ to the other halfspace by setting it identically equal to zero there.

\subsubsection{OU process of Barndorff-Nielsen and Shephard}
To define the volatility process of a Stochastic Volatility (SV) model, O.E.~Barndorff-Nielsen and N.~Shephard \cite{BNandS2001} suggested taking a subordinator $\{\tilde{L}_t,\ t\ge0\}$ (i.e. a L\'evy process with nondecreasing sample paths and hence of almost surely finite total variation on any interval; for more details, see, e.g. \cite[p. 137]{sato} or \cite[Ch. 3]{Bertoin})  as the driving process and considering the SDE
\[
d\tilde{\sigma}^2_t=-\alpha\tilde{\sigma}^2_tdt+d\tilde{L}_{\alpha t},\quad t\ge0, \textrm{where $\alpha>0$ is fixed}.
\]
The initial value $\tilde{\sigma}^2_0$ is considered independent of $\{\tilde{L}_t, t\ge0\}$. The solution of this SDE is the OU-type process
\[
\tilde{\sigma}^2_t=e^{-\alpha t}\left(\int_0^te^{\alpha s}d\tilde{L}_{\alpha s}+\tilde{\sigma}^2_0\right),\quad t\ge0.
\]
From \cite[Thm. 1]{cogarch}, we know that $\tilde{\sigma}_t, t\ge0$ is a time-homogeneous (strong) Markov process, so from Corollary \ref{genOUvar} we immediately get that $v(\tilde{\sigma}^2,[0,T])\le 1$  $\mathbb{P}^x$-almost surely whenever $\tilde{\sigma}^2_0=x$ as $U_t=-\alpha t$ and $L_t=\tilde{L}_{\alpha t}$ are both of finite total variation on any interval. On the other hand, due to the term $\tilde{\sigma}^2_0e^{-\alpha t}=xe^{-\alpha t}$, $v(\tilde{\sigma}^2,[0,T])\ge 1$ $\mathbb{P}^x$-almost surely.

The next step in the SV model of Barndorff-Nielsen and Shephard is to consider the (logarithmic) asset price process $\tilde{G}_t, t\ge0$ as the solution to SDE
\[
d\tilde{G}_t=(\mu+b\tilde{\sigma}^2_t)dt+\tilde{\sigma}_tdW_t,\quad t\ge0,\ \tilde{G}_0=0,
\]
where $\mu$ and $b$ are constants and $\{W_t, t\ge0\}$ is standard Brownian motion, independent of $\tilde{\sigma}^2_0$ and $\{\tilde{L}_{\alpha t}, t\ge0\}$.

The L\'evy-type solution of this equation is
\[
\tilde{G}_t=\mu t+b\int_0^t\tilde{\sigma}^2_sds+\int_0^t\tilde{\sigma}_sdW_s,\quad t\ge0.
\]
Observe that $\tilde{G}$ has a.s. continuous sample paths and nonzero quadratic variation, so it is of a.s. unbounded total variation on any interval, and so $v(\tilde{G},[0,T])\ge1$.

 Consider the trivariate (strong) Markov process $X_t=(t, \tilde{\sigma}_t^2, \tilde{G}_t), t\ge0$ (see \cite[Thm.~1]{cogarch} again) which can be considered as the solution of an SDE by introducing the L\'evy process $Y_t=(t, \tilde{L}_{\alpha t}, W_t)'$, $t\ge0$,
\[
d X_t=\left(
        \begin{array}{ccc}
          1 & 0 & 0 \\
          -\alpha\tilde{\sigma}^2_t & 1 & 0 \\
          \mu+b\tilde{\sigma}^2_t & 0 & \sqrt{\tilde{\sigma}^2_t} \\
        \end{array}
      \right) dY_t=\Phi(X_t)dY_t,\quad X_0=(0,\tilde{\sigma}^2_0,0)'.
\]
It is easy to see that the coefficient $\Phi$ as well as the sequence $(\Phi^n)_{n\in\bbn}$ as defined in the proof of Theorem \ref{thm:levysde} are admissible. Furthermore, the mapping $\xi\to\Phi'(x)\xi$ is a bijection whenever $x\notin\{y=(y_1,y_2,y_3)'\in\bbr^3:y_2=0\}$. From Theorem \ref{thm:levysde} we get
\[\begin{split}
v(\tilde{G},[0,T])&= \max\{1, v(\tilde{\sigma}^2,[0,T]), v(\tilde{G},[0,T])\}\\
&=v(X,[0,T])\le v(Y,[0,T])
\\&=\max\{1, v(\tilde{L}_{(\alpha\cdot)},[0,T]), v(W,[0,T])\}=2.
\end{split}
\]
As for the lower bound, Theorem \ref{thm:levysde} yields
\[
v(X,[0,T])\ge \beta_2(Y)=\max\{1,\beta_2(\tilde{L}_{\alpha \cdot}), 2\}=2,
\]
and so $v(\tilde{G},[0,T])=2$.

An alternative is to use \cite[Thm. 12.8 and Cor. 12.7]{DudleyNorvaisaCFC}), which gives the same result since $\tilde{G}_t$ is a sum of two continuous terms of bounded total variation and a martingale term having continuous paths.

\subsubsection{COGARCH process}

The COGARCH process was defined by C.~Kl\"uppelberg, A.~Lindner and R.~Maller \cite{cogarch} as a generalization of a popular discrete-time GARCH(1,1) process to continuous time. It shares many similarities with the above considered SV-process of Barndorff-Nielsen and Shephard (see, e.g. \cite{KLM2006}). To be exact, let us start by taking a L\'evy process $\{L_t, t\ge0\}$ and two constants $\delta\in(0,1)$ and $\lambda\ge0$. Then one considers an auxiliary L\'evy process
\[
X_t:=-t\log\delta-\sum_{0<s\le t}\log\left(1+(\lambda/\delta)(\Delta L_s)^2\right),\quad t\ge0.
\]
It is well-known that $\{X_t, t\ge0\}$ is a spectrally negative L\'evy process of finite variation, with drift $\gamma_{X,0}=-\log\delta$, zero Gaussian component $\tau_X^2=0$ and the L\'evy measure $\Pi_X$ being the image measure of $\Pi_L$ under the transformation $x\mapsto -\log(1+(\lambda/\delta)x^2), x\in\bbr$ (see \cite[Prop. 3.1]{cogarch}). Next, by taking a $b>0$ and a $\sigma_0^2$ independent of $\{L_t, t\ge0\}$, the authors define the left-continuous volatility process $\{\sigma_t, t\ge0\}$ as the square root of
\[
\sigma_t^2=\left(b\int_0^te^{X_s}ds+\sigma_0^2\right)e^{-X_{t-}},\quad t\ge0,
\]
 And then, finally, the c\`adl\`ag COGARCH process $\{G_t, t\ge0\}$ to model logarithmic asset prices is defined by
\begin{equation}\label{dG_COGARCH}
dG_t=\sigma_tdL_t,\quad t\ge0,\ G_0=0.
\end{equation}
It is known that both $\{\sigma_t^2,t\ge0\}$ and $\{(G_t, \sigma_t), t\ge0\}$ are Markovian (see, e.g. \cite[Thm.~3.2, Cor.~3.1]{cogarch}), moreover,
$\sigma_t^2$ satisfies \cite[Prop. 3.2]{cogarch}
\[
d\sigma_{t+}^2=(b+(\log\delta)\sigma_t^2)dt+\frac{\lambda}{\delta}\sigma_t^2d[L,L]^{disc}_{t},
\]
where $[L,L]^{disc}_t=\sum_{0<s\le t}(\Delta L)_s^2, t\ge0$ is a compound Poisson process which is positive and nondecreasing hence of bounded total variation.

Next define a L\'evy process $Y_t=(t, L_t, [L,L]^{disc}_t)'$ where, contrary to the above investigated cases, the last two coordinate processes are dependent. The process $Y_t$ inherits stochastic continuity as well as stationary and independent increments from $L_t$, and has the L\'evy measure concentrated on the parabola $\{(x_1,x_2,x_3): x_1=0, x_3=x_2^2\}$ in the $(x_2,x_3)$-plane of $\bbr^3$. Hence, by defining $X_t=(t,  G_t, \sigma_{t+}^2), t\ge0$, we have
\[
dX_t=\left(
        \begin{array}{ccc}
          1 & 0 & 0 \\
          0 & \sqrt{\sigma_t^2} &0\\
          b+(\log\delta)\sigma^2_t & 0 & (\lambda/\delta)\sigma_t^2\\
        \end{array}
      \right) dY_t=\Phi(X_{t-})dY_t,\quad X_0=(0,0, \sigma^2_0)'.
\]
So we're in a similar situation as with SV-process of Barndorff-Nielsend and Shephard. Hence, by Theorem \ref{thm:levysde} and Remark \ref{rem:finance},
\[\begin{split}
v(G,[0,T])&\le \max\{1, v(\sigma^2,[0,T]), v(G,[0,T])\}\\
&=v(X,[0,T])\le v(Y,[0,T])
\\&=\max\{1, v(L,[0,T])\}.
\end{split}
\]
In particular, if $L_t$ is taken to be of bounded total variation, then such is also $G_t$. For the lower bound, Theorem \ref{thm:levysde} yields
\[
v(X,[0,T])\ge \beta_2(Y)=\max\{1,\beta_2(L)\},
\]
leaving the possibility that $v(G,[0,T])<1$. The situation is clearer if $L$ possesses linear (resp., Gaussian) component as then $G$ cannot be piecewise constant (due to $\sigma_tdt$ (resp., $\sigma_tdW_t$) term in \eqref{dG_COGARCH}) and hence $v(G,[0,T])\ge1$ almost surely, implying (see Remark \ref{rem:indices_Levy case})
\[
\beta_2(L)\le v(X,[0,T])=v(G,[0,T])\le v(L,[0,T])=\beta_1(L).
\]
 We also mention that the symbol of the COGARCH process was computed in \cite{cogarchsymbol} but we have opted to avoid using it here.




\section{Conclusions}

In this paper, we have investigated sufficient conditions for the boundedness of strong $p$-variation of paths of L\'evy-type processes and illustrated the usefulness of the generalized Blumenthal--Getoor indices defined in terms of the probabilistic symbol of such a process. Our lower index complements these results by allowing for a criterion for the infiniteness of the $p$-variation. The last part of the paper was devoted to examples and applications of our main results (Theorem \ref{thm:firstcrit} and Theorem \ref{thm:lowerbound}). In particular, we have considered  L\'evy-driven SDEs, stable-like and generalized Ornstein--Uhlenbeck processes, and several processes used in mathematical finance. Furthermore, we have shown that it is not possible to extend the notion of the symbol to Hunt semimartingales.

\end{document}